\documentclass[11pt, oneside]{article}   	

\usepackage[a4paper, left=3.2cm,top=3.2cm,right=3.2cm,bottom=3.7cm]{geometry}
\usepackage{concmath}
\usepackage[T1]{fontenc}

\newcommand{\skl}[1]{(#1)}
\newcommand{\cc}{{\hspace{-0.8mm}\text{\scalebox{0.8}{$C$}}}}
\newcommand{\zz}{(\hspace{-0.4mm}z\hspace{-0.3mm})}

\usepackage{graphicx}
\usepackage{epsfig}
\usepackage{caption,subcaption}
\usepackage{amssymb}
\usepackage{amsmath}
\usepackage{tabularx}
\usepackage{booktabs}
\usepackage{authblk}
\usepackage{hyperref}

\usepackage{amsthm}
\newtheorem{theorem}{Theorem}

\newtheorem{corollary}[theorem]{Corollary}

\newtheorem{example}[theorem]{Example}
\newtheorem{assumption}[theorem]{Assumption}

\theoremstyle{definition}

\newtheorem{definition}[theorem]{Definition}

\usepackage{siunitx}
\usepackage{bm}
\usepackage{mathtools}
\usepackage{tikz}

\usepackage{enumitem}
\setitemize{label=\scriptsize{$\blacksquare$},topsep=0em}
\setenumerate{label=(\alph*), topsep=0em}

\newcommand{\al}{\alpha}

\newcommand{\R}{\mathbb R}

\newcommand{\XX}{\mathbb X}
\newcommand{\YY}{\mathbb Y}
\newcommand{\DD}{\mathbb D }
\newcommand{\EE}{\mathbb E }

\newcommand{\MM}{\mathbb M}

\newcommand{\NN}{\mathcal N}

\newcommand{\pr}{\mathbf{P}}
\newcommand{\Fo}{\mathbf{F}}
\newcommand{\Ro}{\mathbf R}
\newcommand{\Go}{\mathbf{G}}

\newcommand{\Ho}{\mathbf{H}}
\newcommand{\reg}{\mathcal{R}}
\newcommand{\tik}{\mathcal{T}}

\newcommand{\nn}{{\bm \Phi}}

\newcommand{\enorm}{\left\|\;\cdot\;\right\|}

\newcommand\norm[1]{\Vert#1\Vert}
\newcommand\set[1]{{\{#1\}}}
\newcommand\lset[1]{\left\{#1\right\}}
\newcommand\snorm[1]{\Vert#1\Vert}

\newcommand{\range}{\operatorname{Im}}
\DeclareMathOperator*{\argmin}{arg\,min}
\DeclareMathOperator{\id}{Id}
\DeclareMathOperator{\fix}{Fix}

\newcommand{\unet}{\mathbf{U}}

\newcommand{\signal}{x}

\newcommand{\data}{y}
\newcommand{\alstar}{{\al^*\skl{\delta,\data^\delta}}}
\newcommand{\deltadata}{_\alstar}

\usepackage{soul}
\usepackage{xcolor}
\colorlet{lred}{red!40}
\colorlet{lgreen}{green!40}
\colorlet{lblue}{blue!40}
\colorlet{lpurple}{purple!40}

\definecolor{lightbluee}{RGB}{181, 210, 247}
\definecolor{lightgreenn}{RGB}{132,228,132}
\definecolor{redd}{RGB}{215,30,30}

\allowdisplaybreaks
\author{Yoeri~E.~Boink}
\affil{Department of Applied Mathematics, University of Twente \authorcr
Multi-Modality Medical Imaging group, University of Twente\authorcr
Drienerlolaan 5, 7522 NB Enschede, Netherlands\authorcr
E-mail: {\tt y.e.boink@utwente.nl}}

\author{Markus~Haltmeier}
\affil{Department of Mathematics, University of Innsbruck\authorcr
Technikerstrasse 13, 6020 Innsbruck, Austria\authorcr
E-mail: {\tt markus.haltmeier@uibk.ac.at}}

\author{Sean Holman}
\affil{School of Mathematics, University of Manchester\authorcr
Oxford Road M13 9PL\authorcr
E-mail: {\tt sean.holman@manchester.ac.uk}}

\author{Johannes~Schwab}
\affil{Department of Mathematics, University of Innsbruck\authorcr
Technikerstrasse 13, 6020 Innsbruck, Austria\authorcr
E-mail: {\tt johannes.schwab@uibk.ac.at}}

\numberwithin{equation}{section}
\numberwithin{figure}{section}
\numberwithin{theorem}{section}

\date{\today}

\title{Data-consistent neural networks for solving nonlinear inverse problems}
			
\begin{document}
\maketitle

\begin{abstract}
Data assisted reconstruction algorithms, incorporating trained neural networks, are a novel paradigm for solving inverse problems. One approach is to first apply a classical reconstruction method and then apply a neural network to improve its solution. Empirical evidence shows that such two-step methods provide high-quality reconstructions, but they lack a convergence analysis. In this paper we formalize the use of such two-step approaches with classical regularization theory. We propose data-consistent neural networks that we combine with classical regularization methods. This yields a data-driven regularization method for which we provide a full convergence analysis with respect to noise. Numerical simulations show that compared to standard two-step deep learning methods, our approach provides better stability with respect to structural changes in the test set, while performing similarly on test data similar to the training set. Our method provides a stable solution of inverse problems that exploits both the known nonlinear forward model as well as the desired solution manifold from data.

\end{abstract}


\section{Introduction}
Let $(\XX, \enorm)$ and $(\YY, \enorm)$ be Banach
spaces \footnote{You can take $\XX$ and $\YY$ as finite dimensional spaces $\R^n$ and $\R^q$ with the Euclidian norm if you are not familiar with infinite dimensional Banach spaces.}  and let  $\Fo \colon \DD  \subseteq \XX \to \YY$
be a continuous, possibly nonlinear, mapping.
In this paper we study the stable solution of the  inverse
problem
\begin{equation} \label{eq:ip}
	\text{ Recover $x \in \DD$ from }  \Fo (\signal) = \data  \,,
\end{equation}
where $\data  \in \Fo(\DD)$ are exact data. We are especially interested in the noisy data case where data $\data^\delta \in \YY$ are given
with  $\norm{\data - \data^\delta} \leq \delta$.
For solving problems of this form we introduce and analyze
convergent regularization methods comprising deep neural networks.

An inversion method for exact data is a right inverse $\Go _0 \colon \Fo(\DD) \rightarrow \DD$ for  $\Fo$,
\begin{equation} \label{eq:inverse}
	\forall  y \in \Fo(\DD)  \colon \quad
	\Fo( \Go_0 (\data) ) =  y  \,.
\end{equation}
The inversion  method  $\Go_0$ therefore recovers
elements in $\Go_0\Fo(\DD) = \fix( \Go_0\Fo )$,
the set of fixed points of  $\Go_0\Fo \colon \DD \to \DD$.
We are mainly interested in the case where
\eqref{eq:ip} is ill-posed, where no continuous right inverse $\Go_0$
exists. If noisy data $\data^\delta \in \YY$ are given
with  $\norm{\Fo (\signal) - \data^\delta} \leq \delta$ for $x \in \Go_0\Fo(\DD)$,
then $\Go_0 (\data^\delta)$   is either not well  defined or arbitrary far away
from $x$. In this  case one has to apply  regularization methods to the data,
which are stable approximations to $\Go_0$ and continuous on all of $\YY$.

\subsection*{Background}
Well established regularization methods are   quadratic Tikhonov regularization \cite{engl1996regularization,Mor84,TikArs77} and
iterative regularization  \cite{BakKok04,KalNeuSch08}.
Both methods are  designed to approximate  minimum norm solutions
$\Go_0(\data) \in \argmin \set{\norm{\signal - \signal_0} \mid \Fo(\signal) =  y}$ with
fixed $\signal_0 \in \XX$.
However, for most applications minimum norm solutions are not the desired ones. One way to overcome  this issue is by convex variational regularization \cite{scherzer2009variational}, where one takes
\begin{equation} \label{eq:nett}
\Go_\alpha (y^ \delta)  \in
 \argmin \lset{ \frac{1}{2} \norm{\Fo (\signal)- \data^\delta}^2
 + \alpha \reg(\signal)  \mid x \in \DD } \,,
\end{equation}
which approximate $\reg$-minimizing solutions
$\Go_0(\data) \in \argmin \set{\reg(\signal) \mid \Fo(\signal) =  y}$. Here the
regularization functional  $\reg \colon \XX \to [0, \infty]$ incorporates a-priori information and takes the role of the norm.
There are still several challenges related  to  variational regularization
techniques. First,  computing  $\reg$-minimizing solutions requires time consuming  iterative minimization schemes. Second,
finding a regularization functional  that well models  solutions  of interest is a difficult issue. Typical choices such as total variation  or the
$\ell^q$-norm with respect to some frame enforce strong handcrafted prior assumptions that are often not met in  practical applications.
One solution for this problem was proposed in \cite{li2020nett,lunz2018adversarial}, where a regularization term in the form of a neural network is learned before it is applied in the classical setting.

Recently, several deep learning methods to solve inverse problems were proposed. Some approaches apply iterative methods, which alternate between data consitency steps and neural network updates, \cite{adler2017solving,chang2017one,kobler2017variational,sun2016deep,kofler2018u,boink2019partially}, while others aim for a fully learned reconstruction scheme \cite{zhu2018image,boink2019learned}. Further, a popular approach for imaging problems is using a neural network as a second step after some initial reconstruction. Several such post processing methods have been considered in the literature \cite{he2016deep,jin2017deep,kobler2017variational,rivenson2017deep}. 
In these two-step approaches, the reconstruction network takes the form $\Ro  = \nn  \Go $ where  $\Go \colon \YY \to \XX$
maps the data to the reconstruction space (reconstruction layer or backprojection) and $\nn \colon \XX \to \XX$ is a neural network (NN) whose free parameters are  adjusted to the training data. In particular, so called residual networks $\nn  = \id_{\XX}  +  \unet$, where only the residual part $\unet$ is trained  \cite{hamilton2018deep,jin2017deep,lee2017deep,antholzer2019deep}
show very accurate  results for solving inverse problems. Due to the huge amount of recent research in this field our discussion of related work is necessarily incomplete, although we have tried to mention the foundational works. 

In order to address the ill-posedness of linear inverse problems,
regularizing networks of the form $\Ro_\al  = \nn_\al  \Go_\al $ were introduced in \cite{schwab2018big,schwab2019deep}. Here
 $\Go_\al\colon \YY \to \XX$  defines any  regularization and
 $\nn_\al \colon \XX  \to \XX$ are trained neural networks
 approximating  data-consistent  networks. In the linear case these networks are called nullspace  networks because they only add parts in the kernel  of $\Fo$ as proposed in  \cite{mardani2017deep,schwab2019deep}. In this paper 
 we   derive convergence and convergence rates for data consistent network families $(\Ro_\alpha)_{\alpha >0}$ for nonlinear problems and provide some numerical examples.

\subsection*{Regularizing networks for nonlinear problems}
Let $\Go_0$ be any right inverse of $\Fo$ and $(\Go_\al)_{\al >0}$
a  regularization of $\Go_0$,  for example classical Tikhonov regularization. Let  $(\nn_\al)_{\al  >0}$
be a family of Lipschitz continuous mappings $\nn_\al \colon \XX \to \XX$.
In this paper we show that under suitable assumptions the reconstruction networks
\begin{equation} \label{eq:recnet}
	\Ro_\al  \coloneqq  \nn_\al   \Go_\al \colon \YY \to \XX \,,
\end{equation}
define a convergent regularization method. Additionally, we derive convergence rates (quantitative error estimates) for the reconstruction error.

A main condition for these results is that
$\nn_\al \Go_\al \Fo$  converges pointwise to a network of the form
  $\nn_0  \Go_0 \Fo$ where   $\nn_0 \colon \XX \to \XX$ is data-consistent in the sense  that $\Fo \nn_0  z  = \Fo z$   for $z \in  \Go_0 \Fo (\DD)$. The latter property implies that $\nn_0$  preserves data consistency of $\Go_0$, meaning that if $\Go_0(\data)$ is a solution of \eqref{eq:ip}, then $\nn_0 \Go_0(\data)$ is a solution of  \eqref{eq:ip} too. Hence the goal of the learned mapping $\nn_0$ is to improve solutions of the inverse problem, while keeping data-consistency of the initial solution. We prove that the family of reconstruction  networks $(\nn_\al \Go_\al)_{\al>0}$  is a convergent regularization method, and we provide convergence rates.

The benefits of $(\nn_\al \Go_\al)_{\al >0}$ over $(\Go_\al)_{\al >0}$
are twofold. First,  in the limit  $\al \to 0$, the network
$\nn_0  \Go_0 \Fo$ selects solutions  in  $\nn_0 \Go_0\Fo(\DD)$
that can be trained to better reflect the desired image class
than $\Go_0\Fo(\DD)$. Second, for   $\al > 0$, the networks $\nn_\al $ can be trained to undo the smoothing   effect of $\Go_\al \Fo $ and thereby  allow for obtaining  convergence rates  for less regular elements than the original regularization method $(\Go_\al)_{\al >0}$. The operator $\nn_0  \Go_0$ can be seen as a right inverse that is learned from a suitable class of training data.

\subsection*{Outline}
The paper is organized as follows. In section 2 we describe the regularization of nonlinear inverse problems and we define the proposed two-step data-driven regularization method. We introduce data-consistent networks, which allow to define the regularization method called \emph{regularizing networks}, which approximate a data-driven right inverse. We investigate under which assumptions these networks generate a convergent regularization method and we give examples of how such regularizing networks can be constructed. In section 3 we present a convergence analysis and derive convergence rates for the proposed method. Section 4 presents the mathematical description of the inverse problems considered in the numerical simulations. These simulations are explained in detail in section 5, after which results are shown in section 6. Additional simulation results can be found in the appendices. The paper concludes with a short summary of the established theory and the numerical simulations.

\section{Convergence of regularizing networks}
Throughout the rest of this paper, let $\Fo \colon \DD  \subseteq \XX \to \YY$ be  a continuous mapping between Banach spaces $\XX$ and $\YY$.
We study the stable solution  of the inverse problem \eqref{eq:ip}.
In this section we introduce the regularizing networks
and present the convergence analysis.

\subsection{Regularization of inverse problems}
Let $\Go_0 \colon \Fo(\DD)  \to \DD$ be any right inverse for $\Fo$.
If \eqref{eq:ip} is ill-posed and $\data^\delta$ are noisy
data with $\norm{\Fo (\signal) - \data^\delta} \leq \delta $,
then the reconstruction method $\Go_0 (\data^\delta)$ is unstable, meaning arbitrarily far away from $\Go_0 \Fo (\signal) $ or not defined.  To obtain meaningful approximations of $\Go_0 \Fo (\signal)$,  one has to apply regularization methods defined as follows.

\begin{definition}[Regularization method]\label{def:regularization}
Let $\skl{\Go_\al}_{\al >0}$ be a
family  of continuous mappings $\Go_\al  \colon \YY \to \XX$. If for all $x \in \Go_0 \Fo(\DD)$ there exists a parameter choice function $\al^* \colon \skl{0, \infty} \times \YY
\to \skl{0, \infty}$ such that
\begin{align*}
0 &=
\lim_{\delta\to 0}  \sup \set{ \al^*\skl{\delta, \data^\delta}
\mid
\data^\delta \in B_\delta (\Fo (\signal) )} \\
0 &= \lim_{\delta\to 0} \sup \set{\snorm{ x  -  \Go_{\al^*\skl{\delta, \data^\delta}} (\data^\delta)}
  \mid \data^\delta \in B_\delta (\Fo (\signal) ) } \,,
\end{align*}
where $B_\delta(\Fo(x))$ is the ball with radius $\delta$ around $\Fo(x)$, we call $(\skl{\Go_\al}_{\al >0}, \al^*)$ a regularization method for $\Go_\al$. If $(\skl{\Go_\al}_{\al >0}, \al^*)$
is a  regularization method for $\Go_0$, we call $\skl{\Go_\al}_{\al >0}$
a regularization of $\Go_0$ and $\al^*$ an admissible parameter choice.
\end{definition}

Probably the best known regularization is quadratic
Tikhonov regularization in Hilbert spaces \cite{engl1996regularization}.
Under the assumption that $\Fo$ is weakly sequentially closed, one shows that there exist solutions of \eqref{eq:ip} with minimal
distance  to a given point $\signal_0 \in \XX$ and that
  \begin{equation} \label{eq:tik}
	\tik_{\alpha, \data^\delta} (\signal)
	\coloneqq
	\frac{1}{2} \norm{\Fo (\signal)- \data^\delta}^2
	+ \frac{\alpha}{2}  \norm{ x- \signal_0 }^2
	\quad \text{ for  } x \in \DD
\end{equation}
has  at least one minimizer.
We can define  $\Go_\alpha (\data^ \delta) $ as any minimizer
of $\tik_{\alpha, \data^\delta}$. If the  solution of \eqref{eq:ip}
with minimal distance to $\signal_0$ is unique and denoted
by $ \Go_0 ( \data )$, then $(\skl{\Go_\al}_{\al >0}, \al^*)$
with a parameter choice satisfying  $\delta^2 / \al^* \to 0$  and $\al^* \to 0$ as $\delta \to 0$ is a  regularization method
for $\Go_0$  \cite{engl1996regularization,Mor84}.

Research indicates that solutions with minimal distance
to a fixed initial guess $\signal_0 \in \XX$  are too simple
in many applications. The use of non-quadratic penalties has demonstrated to often give better  results. Recently, deep learning methods showed outstanding performance. Here solutions are defined by a neural network that maps the given data  to a desired solution.

\subsection{Data-consistent networks}
The first ingredient for constructing regularizing two-step networks are data-consistent networks.
\begin{definition}[Data-consistent network]\label{def:data-consistent}
We call $\nn_0 \colon \XX \to \XX$ a  data-consistent network
if $\nn_0$ is Lipschitz continuous   and $ \forall z \in \Go_0\Fo(\DD) \colon
\Fo \nn_0 (z) = \Fo (z)$.
\end{definition}

In data-consistent networks, if $z \in \Go_0\Fo(\DD)$ is a solution
of \eqref{eq:ip}, then $\nn_0(z)$ is solution of
\eqref{eq:ip} too. In particular, $\nn_0 \Go_0$ is  a right
inverse for $\Fo$ with solution set $\nn_0 \Go_0 \Fo  (\DD)
= \fix(\nn_0 \Go_0 \Fo)$. Data-consistent networks can be constructed by
\begin{equation} \label{eq:consistent}
\nn_0(z)= \pr_{z,0} (\unet (z))\,,
\end{equation}
where $\unet \colon \XX \to  \XX$ is a Lipschitz continuous trained neural network, and
$\pr_{z,0}  \colon \XX \to \XX$ a Lipschitz continuous
mapping with $\pr_{z,0} (\signal) \in  \Fo^{-1}(\Fo(z))  = \{ x \in \DD \mid  \Fo(\signal) =  \Fo(z)\}$. The mapping $\pr_{z,0}$  can be  seen as a generalized  projection  on $\Fo^{-1}(\Fo(z))$. In the special case where $\Fo$ is a linear mapping, $\nn_0(z)$ can be chosen as $\nn_0(z) = z + \pr_{\text{ker}(\Fo)}\unet(z)$, where $\pr_{\text{ker}}$ is the projection on the kernel of $\Fo$ \cite{schwab2019deep}. 

\begin{definition}[Regularizing networks] \label{def:regnet}
We call $(\Ro_\alpha\colon  \YY \to \XX )_{\alpha > 0}$ defined by $\Ro_\alpha \coloneqq \nn_\al \circ \Go_\al$  a family of
regularizing networks if the following hold:
\begin{enumerate}[label=(R\arabic*), leftmargin=3em, topsep=0em, itemsep=0em]
  \item\label{r1} $(\Go_\al \colon \YY \to \XX )_{\al>0}$ is a regularization of $\Go_0$.
  \item\label{r2} $(\nn_\al \colon \XX \to \XX)_{\al>0}$  are
  uniformly $L$-Lipschitz continuous mappings.
  \item\label{r3}
   For some data-consistent network $\nn_0 \colon \XX \to \XX$ we have
  \begin{equation} \label{eq:convergence}
  \forall x \in \DD \colon \lim_{\al \to 0} \nn_\al \Go_\al \Fo (\signal) = \nn_0 \Go_0 \Fo (\signal) \,.
  \end{equation}
  \end{enumerate}
\end{definition}

In practice an important issue is to design networks that converge to a data-consistent limiting network as the noise level goes to zero. 
Next we give examples for a possible strategy to train such networks.

\begin{example}
Let $\signal_1, \dots ,\signal_N$ be training signals and $\data_i=\Fo(\signal_i)$ and respectively $\data_i^\delta$ the corresponding data. Further define the vectors of reconstructions $v\coloneqq(v_1, \dots, v_N)$ and $v^\al\coloneqq(v_1^\al, \dots, v_N^\al)$ where $v_i = \Go_0(\data_i)$ and $v_i^\al = \Go_\al(\data_i^\delta)$. The weights of the neural network are denoted by $\theta\in \Theta$. We write $\nn^\theta$ to depict a neural network with fixed architecture, whose weights have not yet been fixed.
\begin{enumerate}
\item One possible simple approach is to take the networks $\nn_\al \coloneqq \nn_0$ for all $\al>0$, where $\nn_0$ is the network obtained by minimizing the functional 
\begin{align}
\min_\theta \sum_i^N \norm{\nn^\theta(v_i)-\signal_i}^2+\mathcal{R}(\theta)\,.
\end{align}
Here $\mathcal{R}$ denotes some regularization functional for the weights $\theta$ that may be used to ensure a small Lipschitz constant. Clearly, since $\nn_0$ is Lipschitz continuous and $(\Go_\al)_{\al>0}$ is a regularization method, we have the desired limit in \ref{def:regnet} for all $x \in \DD$. Now if the data consistency is incorporated in the network architecture, the condition \ref{r3} is satisfied.
\item A more sophisticated approach is to choose the sequence of networks depending on the regularization parameter $\alpha$. Here the networks $\nn_\alpha$ are obtained by minimizing 
\begin{align*}
\min_\theta \sum_i^N \norm{\nn^\theta(v_i^\al)-\signal_i}^2+\mathcal{R}(\theta).
\end{align*}
To enforce the data consistency of the limiting network $\nn_0$ one could either choose the network architecture to be data-consistent, meaning $\forall \theta \in \Theta \ \forall \signal \in \DD\colon \Fo(\nn^\theta(\signal))=\Fo(\signal)$, or taking networks increasingly data-consistent of the form \begin{equation} \label{eq:consistent2}
\nn_\al(z)	= \pr_{z,\al} \unet_\al (z) \,.
\end{equation}
Here $\unet_\al \colon \XX \to \XX$ is a trained networks and $\pr_{z,\al}$ is a Lipschitz continuous mapping with  $\range ( \pr_{z,\al} ) \subseteq  E_{\al,z} \coloneqq \{x  \mid \norm{\Fo (\signal) -\Fo( z)} \leq r(\al)\}$
with $\lim_{\al \to 0} r(\al)  = 0$. Data-consistency is obtained in the limit. One example for $\pr_{z,\al}$ is the metric projection on $E_{\al,z}$ which is Lipschitz continuous if $E_{\alpha,z}$ is convex.  

Note that in \eqref{eq:consistent2} there are no restrictions on the particular choice of the architecture of the networks $\unet_\al$. 
\item Another network architecture guaranteeing data consistency is given by
\begin{equation}
\nn_\alpha(z) = \nn^{\rm dec}\left(\mathbf{S}_0+\alpha \mathbf{S}_1\right)\nn^{\rm enc}(z)\,.
\end{equation}
Here $\nn^{\rm enc}$ and $\nn^{\rm dec}$ denote an encoder and decoder network respectively, $\mathbf{S}_0$ denotes an $\alpha$-independent network and $\mathbf{S}_1$ denotes a network that is allowed to depend on $\alpha$. 
\end{enumerate}
\end{example}

\subsection{Convergence analysis}
\begin{theorem}[Regularizing networks]
Any family of regularizing networks $(\Ro_\al = \nn_\al \Go_\al)_{\al>0} $ (see Definition \ref{def:regnet}) is a regularization for $\nn_0 \Go_0$ in the sense of
Definition \ref{def:regularization}.
\end{theorem}

\begin{proof}
Let $x \in \nn_0 \Go_0 \Fo (\DD)$, $\data^\delta  \in \YY$ with  $\norm{\Fo (\signal) - \data^\delta} \leq \delta$ and set $\signal_\al^\delta \coloneqq \nn_\al \Go_\al (\data^\delta)$. Then
\begin{align*}
\norm{x  -\signal_\al^\delta}
&= \norm{\nn_0 \Go_0 \Fo (\signal)  - \nn_\al \Go_\al \data^\delta}
\\&\leq
\norm{ \nn_0 \Go_0 \Fo (\signal) - \nn_\al \Go_\al \Fo (\signal)  }+
\norm{\nn_\al \Go_\al \Fo (\signal)   - \nn_\al \Go_\al \data^\delta}
\\&\leq
\norm{\nn_\al \Go_\al \Fo (\signal) - \nn_0 \Go_0 \Fo (\signal) }+
L \norm{\Go_\al \Fo (\signal)   - \Go_\al \data^\delta}\,.
\end{align*}
Now  if $\alstar$ is an admissible parameter
choice for $(\Go_\al)_{\al >0}$, then
\begin{multline}
 \sup \set{\snorm{x -  \Ro_{\al^*\skl{\delta, \data^\delta}} (\data^\delta)}
  \mid \data^\delta \in B_\delta (\Fo (\signal) ) }
  \leq \norm{\nn_\alstar \Go_\alstar \Fo (\signal) - \nn_0 \Go_0 \Fo (\signal) }
  \\
 +  L \sup \set{\snorm{  \Go_\alstar \Fo (\signal)  -  \Go_{\al^*\skl{\delta, \data^\delta}} \data^\delta}
  \mid \data^\delta \in B_\delta (\Fo (\signal) ) } \,.
\end{multline}
According to \ref{r3} in Definition \ref{def:regnet}, the first term converges to zero
and because of
\begin{align*}
    \snorm{  \Go_\alstar \Fo (\signal)  -  \Go_\alstar\data^\delta} \leq
    \snorm{  \Go_\alstar \Fo (\signal)  -  \Go_0 \Fo(\signal)}+\snorm{  \Go_0 \Fo (\signal)  -  \Go_{\alstar}\data^\delta}
\end{align*}
and the fact that
$(\Go_\al)_{\al >0}$ with $\al = \alstar$ is a regularization method for
 $\Go_0$, the second term converges to zero as $\delta  \to 0$.
\end{proof}

\section{Convergence rates}
Another important issue is the rate of approximation. This means specifically that there exists a decreasing function $f \colon  (0, \infty) \rightarrow (0, \infty) $ such that
$\lim_{\delta \rightarrow 0} f(\delta) = 0$ and
$\norm{ \Ro_{\al^*\skl{\delta, \data^\delta}} (\data^\delta) - \signal}  \leq f(\delta)$
uniformly for  all $\data^\delta \in \YY$ with
$\norm{\Fo(\signal) - \data^\delta} \leq \delta$.
%

\subsection{Reconstruction algorithms and convergence rates}
\begin{definition}[Reconstruction error of an algorithm]
 Let $\XX_0 \subseteq \XX$, $\delta >0$  and $\Go  \colon \YY \to \XX$ be a  reconstruction algorithm. We call
\begin{equation}
	\mathcal{E}( \Go, \delta , \XX_0) =
	\sup \set{ \snorm{ \signal - \Go (\data^\delta) }
	\mid \signal  \in \XX_0 \wedge \data^\delta \in \overline{B_\delta(\Fo(\signal))} }
\end{equation}
the  reconstruction error of $\Go$ over $\XX_0$.
\end{definition}

\begin{definition}[Convergence rate of an algorithm]
 Let $\XX_0 \subseteq \XX$, $r \in (0,1]$ and for
 any $\delta >0$, let $\Go^\delta $ be a reconstruction algorithm.
 We say that $(\Go^\delta)_{\delta>0}$
converges at rate $\delta^r$ over $\XX_0$ if
$\mathcal{E}( \Go^\delta, \delta , \XX_0)  = \mathcal O ( \delta^r )$
as $\delta \to 0$.
\end{definition}

The concept of convergence rates in particular applies for reconstruction algorithms defined by regularization methods. In general, no  convergence rate over $\Go_0\Fo(\DD)$ is possible; they require restricting to proper subsets $\XX_0 \subsetneq  \Go_0\Fo(\DD)$ \cite[Proposition 3.11]{engl1996regularization}.

Source conditions define suitable sets $\XX_0$ for classical Tikhonov regularization and related methods based on minimal norm solutions. We investigate the source conditions (transformed source sets) and convergence rates for regularizing networks where  $(\Go_\al)_{\al>0}$ is Tikhonov regularization in Example \ref{ex:tik}.

\subsection{Rates for the regularizing networks}
Our aim is to prove a convergence rate for $\Ro\deltadata$ assuming a convergence rate
for $\Go\deltadata$. Let $(\nn_\al \Go_\al)_{\al>0}$ be a regularizing network and $\al^*$ a parameter choice function. For any $\delta>0$ we define the reconstruction algorithms $\Go^\delta, \Ro^\delta\colon \YY \rightarrow \XX$ and $\nn^\delta\colon \XX\rightarrow \XX$ by
\begin{align*}
\begin{aligned}
\Go^\delta(z)&\coloneqq\Go_{\al^*(\delta,z)}(z),\\
\nn^{\delta,\data^\delta}(x)&\coloneqq\nn_{\al^*(\delta,\data^\delta)}(x),\\
\Ro^\delta(z)&\coloneqq \nn_{\al^*(\delta,z)}\Go_{\al^*(\delta,z)}(z),
\end{aligned}
\end{align*}
for $x\in \XX$ and $z\in\YY$.
\begin{assumption}[Convergence rate conditions]  \label{ass:rateN}
Let $\XX_0 \subseteq \Go_0 \Fo(\DD)$ satisfy the following for some $r \in (0,1]$
\begin{enumerate}[itemsep=0em, topsep=0em, label=(R\arabic*),leftmargin=3em]
\item \label{rateN1}
$\mathcal{E}( \Go^\delta, \delta ,\XX_0)  = \mathcal O ( \delta^r ) $
as $\delta \to 0$.
\item \label{rateN2}
$\sup \set{ \norm{\Go^\delta  ( \data^\delta )  -  \Go^\delta \Fo (x) } \mid x \in \XX_0
\wedge \data^\delta \in \overline{B_\delta (\Fo  (x)  )}}  =
\mathcal O ( \delta^r )  $.
\item \label{rateN3}
$\sup \set{\norm{\nn^{\delta,\data^\delta}(x)-\nn_0(x)}
\mid x \in \XX_0
\wedge \data^\delta \in \overline{B_\delta (\Fo  (x)  )}} = 
\mathcal O ( \delta^r ) $.
\end{enumerate}
\end{assumption}

The first condition \ref{rateN1} means that $(\Go\deltadata)_{\delta >0}$
converges at rate $\delta^r$. Condition \ref{rateN2} is a stability estimate for $\Go\deltadata$. Condition \ref{rateN3} gives a relation between $\nn\deltadata$, applied in noisy cases, and $\nn_0$, applied in the noiseless case.

\begin{theorem}[Convergence rate for regularizing networks]\label{thm:rate-rn}
Let $\MM_0 = \nn_0(\XX_0)$. Under Assumption~\ref{ass:rateN} we have
$\mathcal{E}( \Ro^\delta, \delta ,\MM_0)  = \mathcal O ( \delta^r )$.
\end{theorem}

\begin{proof}
Let $\signal\in \MM_0$, $\norm{\Fo (\signal) -\data^\delta} \leq \delta$, and $z \in \XX_0$ s.t. $\nn_0(z)=\signal$.  Then
\begin{align*}
 \|  \Ro^\delta  (\data^\delta)  - \signal  \|
 &\leq
\snorm{ \Ro^\delta  (\data^\delta)  - \Ro^\delta\Fo (\signal)}
+ \snorm{ \Ro^\delta \Fo (\signal)- \signal}
\\
 &\leq
L \snorm{ \Go^\delta (\data^\delta)  -  \Go^\delta \Fo (\signal)}
+ \snorm{ \nn_0 \Go^\delta \Fo(z) - \nn_0(z)}
 \\&
 \hspace{0.08\textwidth} +
\snorm{ \Ro^\delta \Fo (\signal) - \nn_0 \Go^\delta \Fo(z)}
 \\
 & \leq 
L \snorm{ \Go^\delta  (\data^\delta)  -  \Go^\delta \Fo (\signal)}
+ L \snorm{ \Go^\delta \Fo(z) - z} \\&
 \hspace{0.08\textwidth} +
\snorm{ \nn^{\delta,\Fo(x)} \Go^\delta \Fo(\signal) - \nn_0 \Go^\delta \Fo(x)}
 \\
 & \leq 
L \snorm{ \Go^\delta  (\data^\delta)  -  \Go^\delta \Fo (\signal)}
+ L \snorm{ \Go^\delta \Fo(z) - z} \\&
 \hspace{0.08\textwidth} +
 \snorm{ \nn^{\delta,\Fo(\signal)} \Go^\delta \Fo(\signal) - \nn^{\delta,\Fo(\signal)} \Go_0 \Fo(\signal)}\\&
 \hspace{0.08\textwidth} + 
 \snorm{ \nn^{\delta,\Fo(\signal)} \Go_0 \Fo(\signal) - \nn_0 \Go_0 \Fo(\signal)}\\&
 \hspace{0.08\textwidth} +
 \snorm{ \nn_0 \Go_0 \Fo(\signal) - \nn_0 \Go^\delta \Fo(\signal)} 
 \\
 & \leq 
L \snorm{ \Go^\delta  (\data^\delta) - \Go^\delta \Fo (\signal)}
+ L \snorm{ \Go^\delta \Fo(z) - z} \\&
 \hspace{0.08\textwidth} +
 \snorm{ \nn^{\delta,\Fo(x)} \Go_0 \Fo(\signal) - \nn_0 \Go_0 \Fo(\signal)}
 + 2L \snorm{ \Go^\delta \Fo(\signal) - \Go_0 \Fo(\signal)}.
 \end{align*}
 
Each of the above terms are $\mathcal  O ( \delta^r )$: the first term due to the stability estimate \ref{rateN2}, the second term due to \ref{rateN1}, and the third term due to \ref{rateN3}. For the fourth term we use that $\Go_0 \Fo(\signal) = \Go_0\Fo(z)=\Go_0\Fo\Go_0\Fo(w)=\Go_0\Fo(w)=z\in\XX_0$ for some $w\in\DD$. This implies that the fourth term is $\mathcal  O ( \delta^r )$, due to \ref{rateN1} again.
\end{proof}

In the following we give an explicit example of a classical regularization method combined with a sequence of regularizing networks, where Assumption \ref{ass:rateN} is satisfied and therefore Theorem \ref{thm:rate-rn} can be applied.

\begin{example}[Regularizing networks combined with Tikhonov regularization]\label{ex:tik}
Given a G\^ateaux differentiable forward operator $\Fo$ we consider $(\Go_\al)_{\al >0}$ defined by  classical
Tikhonov regularization \eqref{eq:tik}, a data-consistent network $\nn_0 \colon \XX \to \XX$ and a sequence of regularizing networks $(\nn_\al\colon \XX \to \XX)_{\al>0}$ satisfying \ref{rateN3} for $r=1/2$.

\begin{corollary}
If we consider the set $\MM_0\coloneqq \nn_0(\XX_0)$ where $\XX_0$ is the source set of classical Tikhonov reguarlization and we assume that the networks $\nn_\alstar$ converge pointwise to $\nn_0$ on $\XX$ at rate $\mathcal O( \delta^{1/2} )$ as $\delta\rightarrow 0$, then $\mathcal{E}( \Ro^\delta, \delta ,\MM_0 )  = \mathcal O ( \delta^{1/2} )$
\end{corollary}

\begin{proof}
The convergence rate condition \ref{rateN1} holds according to \cite{engl1996regularization}. Since for $x= \nn_0(z) \in \MM_0$ we have $\Fo(x) = \Fo(z)$ (Definition \ref{def:data-consistent}) and because of the  stability estimate for Tikhonov regularization \cite{engl1996regularization} we have \begin{multline}
\sup \set{ \norm{\Go^\delta  ( \data^\delta )  -  \Go^\delta (\Fo  (x)) } \mid x \in \MM_0
\wedge \data^\delta \in \overline{B_\delta (\Fo  (x)  )}}   \\
=
\sup \set{ \norm{\Go^\delta  ( \data^\delta )  -  \Go^\delta (\Fo  (z)) } \mid z \in \XX_0
\wedge \data^\delta \in \overline{B_\delta (\Fo  (x)  )}}
  = \mathcal O( \delta^{1/2} ) \,,
\end{multline}
which shows \ref{rateN2}.
Finally \ref{rateN3} holds by assumption and therefore the conditions of 
Theorem~\ref{thm:rate-rn} are satisfied for $r=1/2$.
\end{proof}
\end{example}

This shows one of the main benefits of the concept of regularizing networks, namely  transforming the set $\XX_0$ on which the basic regularization converges at a certain rate, to a different data dependent set $\nn_0(\XX_0)$ with possibly less regularity, while preserving the convergence rate.

\section{Considered inverse problems}\label{sec:gen_IP}
In this section we provide the general mathematical description of inverse problems that are considered in the simulation experiments. We define nonlinear mappings $\Fo$ and derive formulations of right inverses $\Go_0$ and data-consistent networks $\nn_0$. 

\subsection{Projection on convex set}\label{sec:projection}
We consider the nonlinear inverse problem, where $\Fo:=\pr_C\colon\DD\to C$ is a metric projection on a closed convex set $C\subseteq\DD$, i.e.
\begin{align}\label{eq:proj}
\data = \pr_C(x):=\argmin_{\bar{x}\in C} \lset{\norm{\bar{x}-x}}\,.
\end{align}
The affine normal cone to $C$ at $x$ is defined as
\begin{align*}
\NN_C(x):=\lset{\tilde{x}\in\DD~|~\pr_C(\tilde{x})=x}\,.
\end{align*}
It is easily shown that any mapping that maps $x\in C$ to an element in the normal cone $\NN_C(x)$ is a right inverse of $\pr_C$. In particular, the projection $\pr_{\NN_C(x)}\colon\DD\to\DD$,
\begin{align}\label{eq:projcone}
\pr_{\NN_C(x)}(\hat{x}):=\argmin_{\tilde{x}\in \NN_C(x)\cap \DD} \lset{\norm{\tilde{x}-\hat{x}}},
\end{align}
combined with any function $\Ho_0\colon C\rightarrow \DD$ defines a right inverse $\Go_0\colon C \rightarrow \DD\colon \signal\mapsto \pr_{\NN_C(x)}(\Ho_0(x))$. This follows because $\forall x \in C\colon \pr_C(\pr_{\NN_C(x)}(\Ho_0(x)))=x$. We assume, that \eqref{eq:projcone} is well defined for all $x\in C$ and $\hat{x}\in\DD$. According to Definition \ref{def:data-consistent}, a data-consistent network $\nn_0$ satisfies $\forall z\in\Go_0\Fo(\DD)\colon\Fo\nn_0(z)=\Fo(z)$. We define 
\begin{align}\label{eq:proj_data_inv}
\nn_0(z) := \pr_{\NN_C(\pr_C(z))}(\unet(z)),
\end{align}
so this requirement is satisfied. Here $\unet \colon \XX \to  \XX$ is any Lipschitz continuous trained neural network (c.f. definition \ref{def:data-consistent}). See Figure \ref{fig:projection} for a visual illustration.

\begin{figure}[!ht]
    \centering
    \resizebox{0.5\textwidth}{!}{%
    \begin{tikzpicture}
    \draw[fill=lightbluee, line width=1.5] (0:2 cm) -- (72:2 cm) -- (144:2 cm) -- (216:2 cm) -- (288:2 cm) -- cycle;

\draw [top color=lightgreenn, bottom color=white, shading angle=90, draw=none] ([shift=(0:2cm)] -36:3cm ) -- ([shift=(0:2cm)] 0:0cm ) -- ([shift=(0:2cm)] 36:3cm ) -- cycle;
\draw [line width=1.5] ([shift=(0:2cm)] -36:3cm ) -- ([shift=(0:2cm)] 0:0cm ) -- ([shift=(0:2cm)] 36:3cm );

\node[scale=2] at (0,0) {$C$};
\node[scale=1.7] at (5.4cm,0cm) {$\mathcal{N}_{\hspace{0.5mm}\cc}(\hspace{-0.4mm}\mathbf{P}_\cc\zz\hspace{-0.3mm})$};
\node [fill=redd, draw=black, shape=circle, inner sep=1.45pt] at (3.6cm,0.6cm) {};
\node[scale=1.4] at (3.85cm,0.65cm) {$z$};
\node [fill=redd, draw=black, shape=circle, inner sep=1.55pt] at (2.0cm,0.0cm) {};
\node[scale=1.4, rotate=38] at (2.35cm,0.68cm) {$\mathbf{P}_\cc\zz$};
\node [fill=redd, draw=black, shape=circle, inner sep=1.55pt] at (3.2cm,-1.6cm) {};
\node[scale=1.4] at (3.25cm,-2.0cm) {$U\zz$};
\node [fill=redd, draw=black, shape=circle, inner sep=1.55pt] at (3.54cm,-1.13cm) {};
\node[scale=1.4] at (5.3cm,-1.0cm) {$\nn_0\zz$};
\draw[redd,dashed, line width=1.0] (3.6cm,-0.3cm) arc (-60:60:0.5cm);
\draw[redd,dashed, line width=1.0, ->, >=latex] (3.685cm,-0.245cm) arc (120:195:1.10cm);
\draw[redd, line width=1.0, ->, >=latex] (3.26cm,-1.51cm) -- (3.47cm,-1.22cm);
    \end{tikzpicture}}
    \caption{Visualization of the data-consistent network for the projection problem. It can be seen that $\pr_C(\nn_0(z)) = \pr_C(z)$, as required by the definition of a data-consistent network.}
    \label{fig:projection}%
\end{figure}
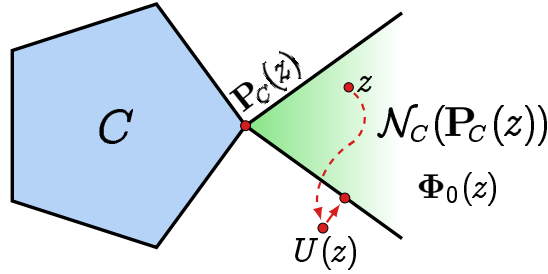

\subsection{Composition of mappings}\label{sec:composition}
As a second inverse problem, we consider the mapping $\Fo\colon\DD\to\YY$ that is defined as a composition of two (possibly nonlinear) mappings:
\begin{align*}
&\Fo(x):=\Fo_2(\Fo_1(x)), \quad \text{ where }\Fo_1\colon\DD\to\EE \text{ and }\Fo_2\colon\EE\to\YY\,,
\end{align*}
where $\EE$ is a Banach space. Furthermore we impose the restriction that the mapping $\Fo_2$ provides a data-consistent network that can be written as a projection
\begin{align*}
\nn^{(2)}_0(z) := \pr_{S\cap \range(\Fo_1)}(\unet_2(z)),
\end{align*}
where $S\subseteq\EE$. In particular, this is true for the mapping described in Section \ref{sec:projection}, when the projection on a normal cone also maps into the range of the operator $\Fo_1$. The projection onto the intersection in the data-consistent network can be implemented by an alternating projection algorithm \ref{sec:sat_rad}. If we assume that $S\cap\range(\Fo_1)\neq\emptyset$, then we define the data-consistent network $\nn_0$ for the full mapping as
\begin{align}\label{eq:conc_data_inv}
\nn_0(z) = \nn_0^{(1)}\Go_0^{(1)}\nn_0^{(2)}\Go_0^{(2)}\Fo_2\Fo_1(z)
\end{align} 
where $\nn_0^{(i)}$ and $\Go_0^{(i)}$ are defined as the data-consistent network for $\Fo_i$ and the right inverse for $\Fo_i$ respectively. We check the data-consistent property (Definition \ref{def:data-consistent}) by
\begin{align*}
\Fo\nn_0(z) &= \Fo_2\Fo_1\nn_0^{(1)}\Go_0^{(1)}\nn_0^{(2)}\Go_0^{(2)}\Fo_2\Fo_1(z)\\
			&= \Fo_2\Fo_1\Go_0^{(1)}\nn_0^{(2)}\Go_0^{(2)}\Fo_2\Fo_1(z)\\
			&= \Fo_2\nn_0^{(2)}\Go_0^{(2)}\Fo_2\Fo_1(z)\\
			&= \Fo_2\Go_0^{(2)}\Fo_2\Fo_1(z)\\
			&= \Fo_2\Fo_1(z)\\
			&= \Fo(z),\\
\end{align*}
where we used in order: the data-consistent property of $\nn_0^{(1)}$; the definition of a right-inverse $\Go_0^{(1)}$ in combination with the projection on the range of $\Fo_1$; the data-consistent property of $\nn_0^{(2)}$; and the definition of the right-inverse $\Go_0^{(2)}$.

We note that this is not the only data-consistent network possible for such an inverse problem: one could also design a network that only makes use of either $\nn_0^{(1)}$ or $\nn_0^{(2)}$. However, \eqref{eq:conc_data_inv} provides a network that is intuitively clear: an initial solution $z$ is obtained by a classical regularization method, after which first a better `guess' is made by applying a neural network on $\EE$, followed by a neural network that makes a better guess on the reconstruction space $\DD$, while keeping the solutions data-consistent throughout.

\section{Simulation experiments}
In this section, we first specify two examples of the inverse problems described in section \ref{sec:gen_IP}. After that, we explain all neural networks that will be compared for these examples, among which are the derived data-consistent networks. Finally, we provide the implementation details for all simulation experiments.

\subsection{Spatially dependent saturation of multivariate Gaussians}\label{sec:sat_gauss}
In the first simulation experiment, we consider the inverse problem of recovering images of multivariate Gaussians which have been nonhomogeneously saturated. Formally we define the domain $\DD=\XX:=\ell^2(\Omega)$ and we define the saturation mapping as a projection on a convex set, as described in section \ref{sec:projection}. This means $\Fo(x) :=\pr_C(x)$, where 
\begin{align*}
C:=\lset{x\in\ell^2(\Omega)~|~x(r)\leq M(r),~~\forall r\in\Omega,}
\end{align*}
where $M(r)\geq0$ is the saturation value at location $r$. The corresponding right inverse for $y\in C$ is defined as $\Go_0(y) = \pr_{\NN_C(y)}(y) = \id(y)$. The projection defined by \eqref{eq:proj} is explicitly given by
\begin{align*}
\big[\pr_C(x)\big](r):= \min\lset{x(r), M(r)}.
\end{align*}
Since $\NN_C(\pr_C(z))=\lset{x~|~\pr_C(x)=\pr_C(z)}$, \eqref{eq:proj_data_inv} can be written pointwise as 
\begin{align}\label{eq:data_cons_satur}
\big[\nn_0(z)\big](r) = 
\begin{cases}
z(r)	&\text{for }z(r)<M(r),\\
\max\lset{\big[\unet(z)\big](r),M(r)} &\text{for }z(r)\geq M(r).
\end{cases}
\end{align}
We consider the square domain $\Omega:=[-1,1]\times[-1,1]$. The spatially dependent saturation function is defined 
\begin{align}\label{eq:saturation}
M(r) := \begin{cases}0.6 &\text{if }\norm{r}\leq \frac12,\\
					0 	&\text{if }\norm{r}>\frac12. \end{cases}
\end{align}

Each image in the training or test set contains one centered multivariate Gaussian with diagonal covariance matrix, having standard deviations $(\sigma_1,\sigma_2)$ independently randomly chosen in the interval $[0.24,0.32]$. All images in the training and test set are scaled to obtain maximum values randomly chosen in the interval $[0.75,1]$. Opposed to standard networks, one of the benefits of using a data-consistent network is that it is more robust to changes in the data. For this reason, a modified test set has been created, where the Gaussians have standard deviations in the interval $[0.12,0.20]$ with maximum intensities in the interval $[0.6,0.8]$. For the numerical implementation we consider the discretized domain $\bar\Omega:=\R^{128\times128}$ as discretization of $\Omega$.

The data-consistent network $\nn_0(z)$, as described in \eqref{eq:data_cons_satur}, is compared with the neural network $\unet(z)$ without data-consistency. We compare reconstruction quality for both the regular test set and modified test set. A description of the neural network architecture and training details are provided in section \ref{sec:implementation}.

\subsection{Saturation of Radon transformed human chest images}\label{sec:sat_rad}
In the second simulation experiment, we consider the inverse problem of reconstructing images of the human chest, from saturated and highly limited angle Radon measurements. We consider the composition of two mappings as described in section \ref{sec:composition}, where $\Fo_1$ is a linear mapping that acts as the discrete Radon transform and $\Fo_2$ is a nonlinear saturation mapping that saturates the Radon signals at a constant value $M$. For conciseness, we define our domains and mappings in the discretized setup: $\DD:=\R^{n_x\times n_x}$, $\YY=\R^{n_\alpha\times\frac{3}{2}n_x}$, where $n_x=192$ is the number of pixels in each direction of the image and $n_\alpha = 8$ is the number of angles in the Radon transform, uniformly sampled in the interval $[0,\pi]$. 

We now define all elements that are needed to obtain the data-consistent network \eqref{eq:conc_data_inv}, which we repeat here for completeness:
\begin{align}
\nn_0(z) = \nn_0^{(1)}\Go_0^{(1)}\nn_0^{(2)}\Go_0^{(2)}\Fo_2\Fo_1(z).
\end{align} 
A matrix representation $\Fo_1\in\R^{n_\alpha\cdot\frac{3}{2}n_x\times n_x^2}$ of the Radon transform is obtained as described in \cite{lewitt1990multidimensional}. Its right inverse is taken as the pseudo-inverse of $\Fo_1$, i.e. $\Go_0^{(1)}:= \Fo_1^\dag$. Since the mapping is linear, the corresponding data-consistent network is a null-space network \cite{schwab2019deep}, i.e. $\nn_0^{(1)}(z) = z + \pr_{\text{ker}(\Fo_1)}\unet(z)$, where $\pr_{\text{ker}(\Fo_1)} = \id - \Fo_1^\dag\Fo_1$. The saturation mapping $\Fo_2=\pr_C$, its right inverse $\Go_0^{(2)}$ and the data-consistent network $\nn_0^{(2)}$ are chosen as described in section \ref{sec:sat_gauss}, this time with constant saturation level $M=8$. Finally, for this particular choice of $\Fo_2$, the projection on the intersection of convex sets reads $\pr_{\NN_C(\pr_C(z))\cap\range(\Fo_1)}$, which can be achieved by the `projection onto convex sets' (POCS) algorithm \cite{bauschke1996projection}: by alternatingly performing $\pr_{\NN_C(\pr_C(z))}$ and $\pr_{\range(\Fo_1)} = \Fo_1\Fo_1^\dag$, the resulting iteration converges linearly to a point on the intersection.

Training and test images were obtained from the LoDoPaB-CT dataset \cite{leuschner2019lodopab}, which on its turn makes use of the LIDC/IDRI dataset \cite{armato2011lung}. In our work, we only make use of the high quality CT reconstructions in the LoDoPaB-CT dataset that we use as `ground truth' for our setup. The images are scaled to $192\times192$ pixels, after which the mappings $\Fo_1$ and $\Fo_2$ are applied to obtain simulated sinograms. After that, pseudo-inverses $\Go_0^{(2)}$ and $\Go_0^{(1)}$ are applied to obtain the input for our data-consistent network. For this simulation experiment we have also created a modified test set to investigate how the trained networks generalize towards slightly modified data. For conciseness, the procedure to get from the regular test data to the modified test data is not explained in full detail. In short, the test set consists of images in the range of $\Fo_1^\dag$ that produce sinograms that have a maximum below or around the saturation level. This means that the saturation mapping $\Fo_2$ will not have a big effect on the unsaturated sinograms. Images in the modified test set look very similar to the ones in the regular test set, but they often show a small gradient at locations where the regular images show a piecewise constant structure. Some samples from the modified test set are shown in Figure \ref{fig:modified-} and appendix \ref{app:modified}.

The data-consistent network $\nn_0(z)$, as described in \eqref{eq:conc_data_inv}, is compared with two other networks: the first one applies a single neural network to the pseudo-inverse reconstruction; the second one first applies a neural network in the sinogram domain, then applies the pseudo-inverse of the Radon-transform, followed by a neural network in the image domain. For completeness, we summarize the three networks below:

\begin{itemize}
\item \makebox[43mm]{One neural network:\hfill} $\mathbf{N}_1(z) = \unet_1(z).$
\item \makebox[43mm]{Two neural networks:\hfill} $\mathbf{N}_2(z) = \unet_1\Go_0^{(1)}\unet_2\Fo_1(z).$
\item \makebox[43.4mm]{Data-consistent network:\hfill} $\nn_0(z) = \nn_0^{(1)}\Go_0^{(1)}\nn_0^{(2)}\Go_0^{(2)}\Fo_2\Fo_1(z).$
\end{itemize}
We emphasize that the data-consistent networks, in terms of architecture, make use of the same neural networks $\unet_1$ and $\unet_2$ as the first two networks, i.e. $\nn_0^{(1)}(z) = z + \pr_{\text{ker}(\Fo_1)}\unet_1(z)$ and $\nn_0^{(2)}$ makes use of $\unet_2$ as defined in \eqref{eq:data_cons_satur}. A more detailed description of the neural network architecture and training details are provided in section \ref{sec:implementation}.

Ideally, the data-consistent network is trained `end-to-end', meaning that both $\unet_1$ which is used in $\nn_0^{(1)}$, and $\unet_2$ which is used in $\nn_0^{(2)}$, are trained at the same time. However, the application of the POCS algorithm is computationally intensive, since it requires iterative application of the mappings $\Fo_1$ and $\Fo_1^\dag$. For this reason, we have chosen to first train $\nn_0^{(2)}$ to output the unsaturated sinogram, then perform the POCS algorithm and finally train $\nn_0^{(1)}$ to output the reconstructed image.

\subsection{Neural network architecture and training details}\label{sec:implementation}
In this work, the popular U-Net \cite{ronneberger2015unet, jin2017deep} is implemented as a neural network. By using standard nonlinearities such as rectified linear units (ReLUs) and convolutions, the network is Lipschitz continuous, which is a requirement as described in Definition \ref{def:data-consistent}. The Lipschitz constant can be controlled by weight regularization, such as adding an $L^2$-loss on the weights in the loss function.

\begin{table}[!ht]
\begin{center}
\def\arraystretch{1.05}
\scalebox{0.98}[1.0]
{\begin{tabularx}{1.02\textwidth}{l *{3}{X}}
\toprule			 
							& Exp. 1 $(\unet/\nn_0)$:	& Exp. 2 $(\unet_1/\nn_0^{(1)})$: 	& Exp. 2 $(\unet_2/\nn_0^{(2)})$:   \\
							& image domain				& image domain						& sinogram domain 					\\
\midrule
$\#$training samples		& 1024						& 35584								& 35584								\\ 
$\#$validation samples		& 256						& 3522								& 3522								\\ 
$\#$test samples			& 1024						& 3553								& 3553								\\ 
\midrule
depth						& 4							& 4									& 4									\\ 
width						& 2							& 2									& 2									\\ 
$\#$channels in top layer 	& 8							& 16								& 16								\\
convolution size 			& $3\times3$				& $3\times3$						& $3\times3$						\\
nonlinearity				& ReLU						& ReLU								& ReLU								\\
\midrule
start learning rate			& $10^{-3}$					& $10^{-3}$							& $10^{-3}$							\\
final learning rate			& $10^{-4}$					& $2\cdot10^{-4}$					& $2\cdot10^{-4}$					\\
batch size					& 64						& 32								& 32								\\ 
$\#$epochs					& 1000						& 25								& 25								\\ 
\bottomrule
\end{tabularx}}
\end{center}
\caption{U-Net parameter details for all simulation experiments.}
\label{tab:parameters}
\end{table}

The U-Net was implemented as described in \cite{jin2017deep}, although for each experiment some parameters were chosen slightly different to obtain optimal results. For all experiments, the network has a `depth' of four, meaning four times max-pooling and upsampling. The U-Nets in the image domain perform the regular max-pooling and upsampling in two directions, while the U-Net in the sinogram domain performs these only in one direction, leaving the number of angles constant at 8. This is done because the neighboring angles in the sinogram show very little resemblance to each other and there are only 8. The `width', or the amount of convolutions at every depth is chosen to be two. As in \cite{jin2017deep} the number of convolution channels doubles after each max-pooling; the number of channels at the start is stated in Table \ref{tab:parameters}, since this was chosen differently for every simulation experiment. In all experiments a residual structure that is also apparent in \cite{jin2017deep} is used. The U-Net uses $3\times3$ convolutions with biases and applied a ReLU-activation after each convolution, except the last one. In all experiments, an $L^2$-loss function on the difference between output and ground truth is minimized. For optimization, the ADAM optimizer with exponentially decaying learning rate is chosen. The learning rates and batch sizes are stated in Table \ref{tab:parameters}. 

\section{Numerical results}
In this section, the reconstruction quality of data-consistent networks is compared with the U-Nets that are not data-consistent. Besides visual comparison, the quality will be compared by means of peak signal-to-noise ratio (PSNR) and structural similarity (SSIM). For both experiments this will be done for the regular test set as well as the modified test set in order to investigate the generalization capacity of the networks. 

\subsection{Spatially dependent saturation of multivariate Gaussians}
Multivariate Gaussians are saturated with a spatially dependent saturation function \eqref{eq:saturation}, as described in section \ref{sec:sat_gauss}. Results for one sample from the regular test set are shown in Figure \ref{fig:regular}. Here it can be seen that both U-Net and the data-consistent network provide a very accurate reconstruction. This is also reflected in the PSNR and SSIM values shown in Table \ref{tab:quality_gaussians}. The pseudo-inverse reconstruction, which in this case is just the measurement, is not a good one, since a lot of information is lost by applying the saturation mapping.

\begin{figure}[!ht]
\centering
\begin{subfigure}[t]{0.21\textwidth}
\centering
\captionsetup{justification=centering}
\includegraphics[width=0.8\textwidth]{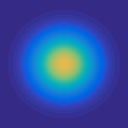}
\includegraphics[width=\textwidth, trim = {61 0 61 0}, clip]{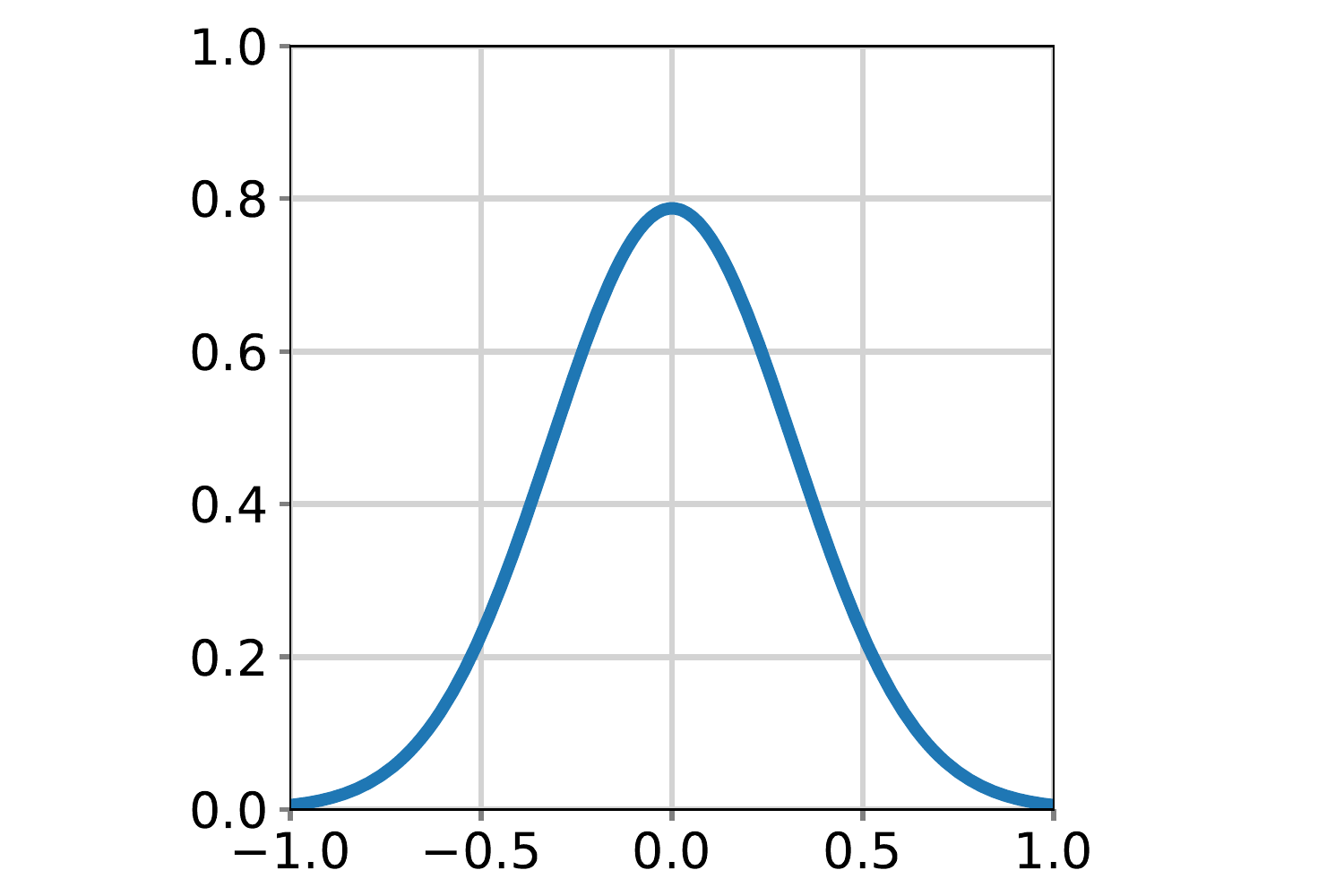}
\caption{Ground truth}
\label{fig:regular-x}
\end{subfigure}
\begin{subfigure}[t]{0.21\textwidth}
\centering
\captionsetup{justification=centering}
\includegraphics[width=0.8\textwidth]{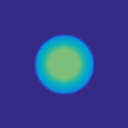}
\includegraphics[width=\textwidth, trim = {61 0 61 0}, clip]{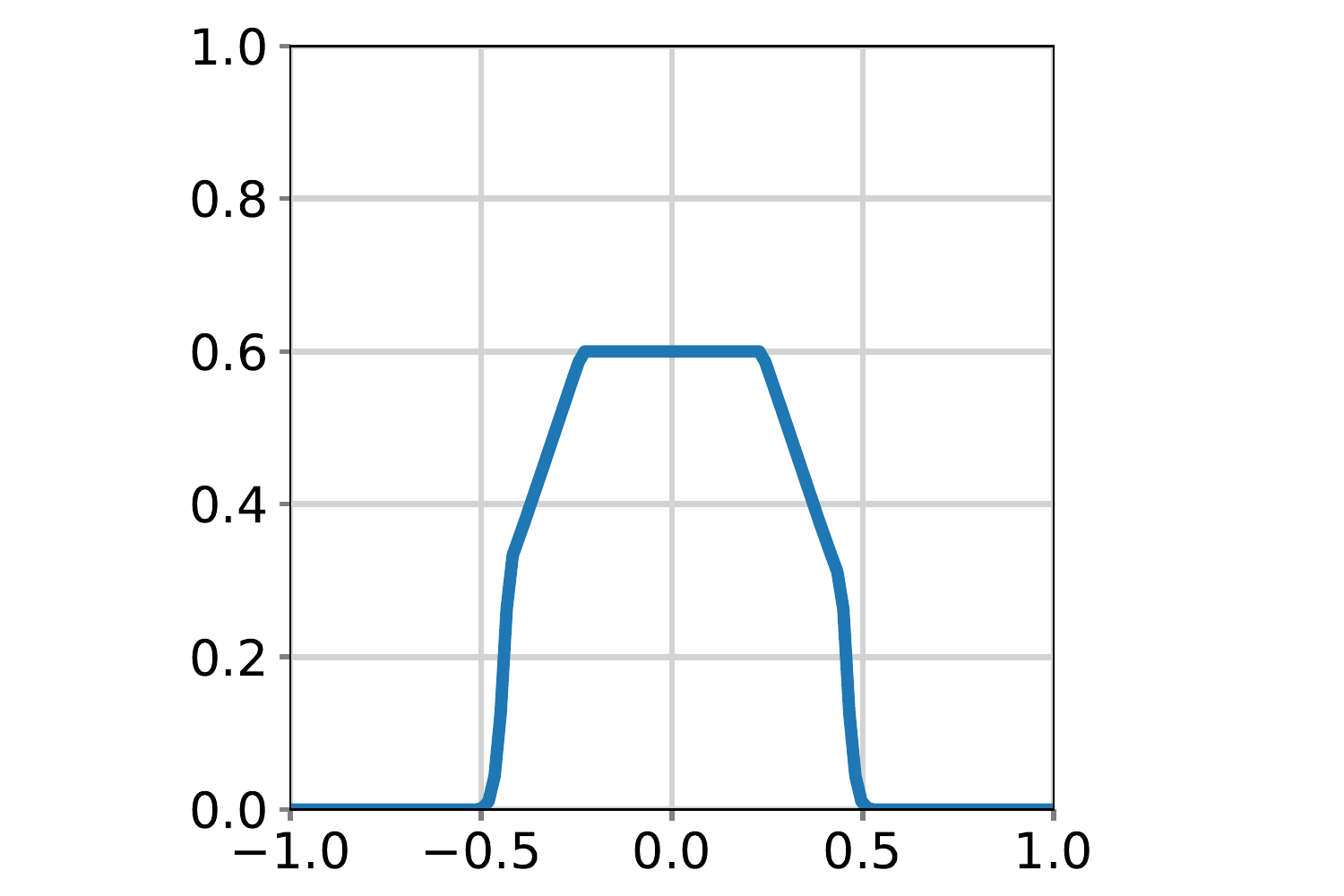}
\caption{Saturated}
\label{fig:regular-y}
\end{subfigure}
\begin{subfigure}[t]{0.21\textwidth}
\centering
\captionsetup{justification=centering}
\includegraphics[width=0.8\textwidth]{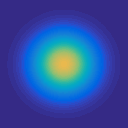}
\includegraphics[width=\textwidth, trim = {61 0 61 0}, clip]{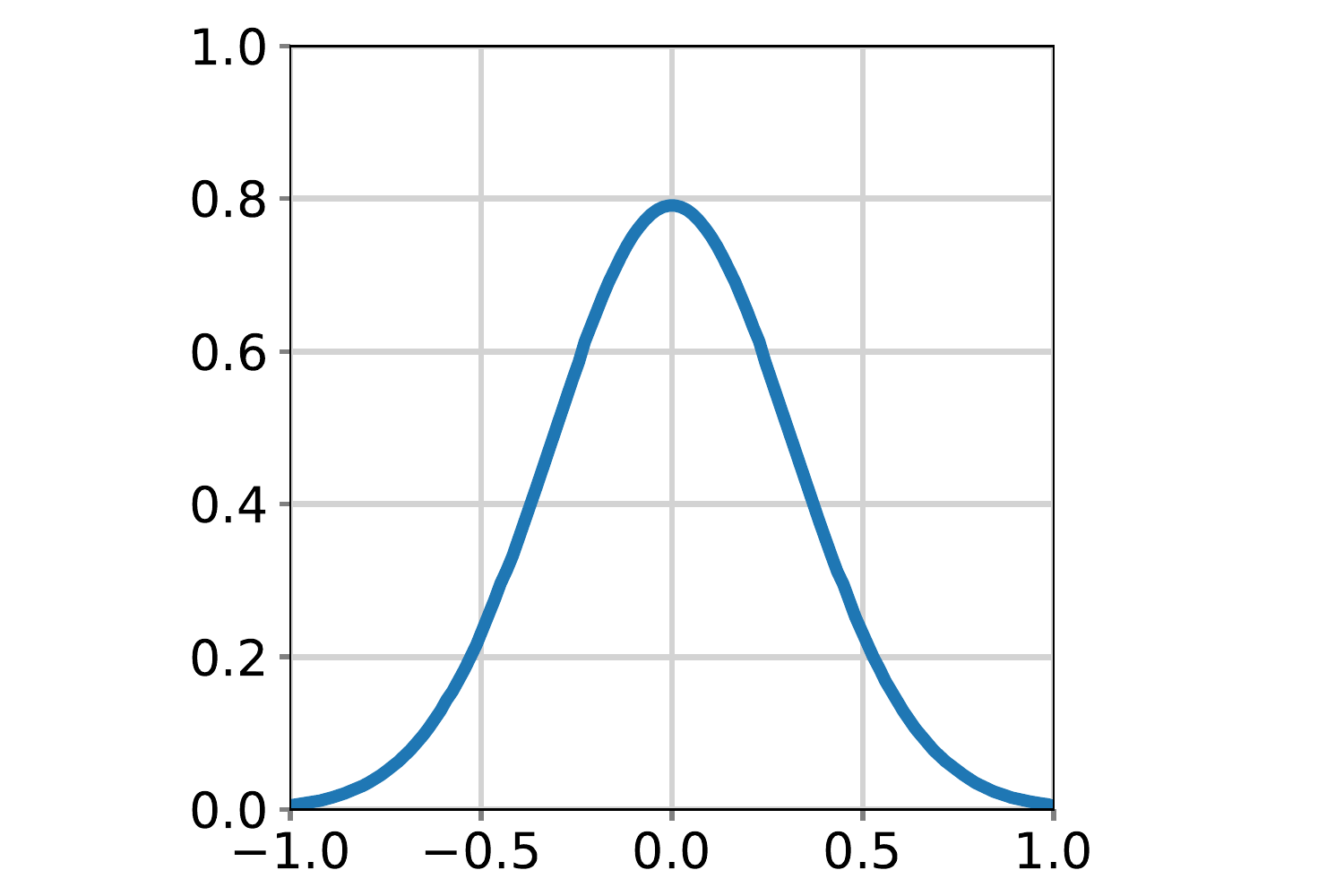}
\caption{U-Net}
\label{fig:regular-u}
\end{subfigure}
\begin{subfigure}[t]{0.21\textwidth}
\centering
\captionsetup{justification=centering}
\includegraphics[width=0.8\textwidth]{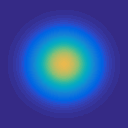}
\includegraphics[width=\textwidth, trim = {61 0 61 0}, clip]{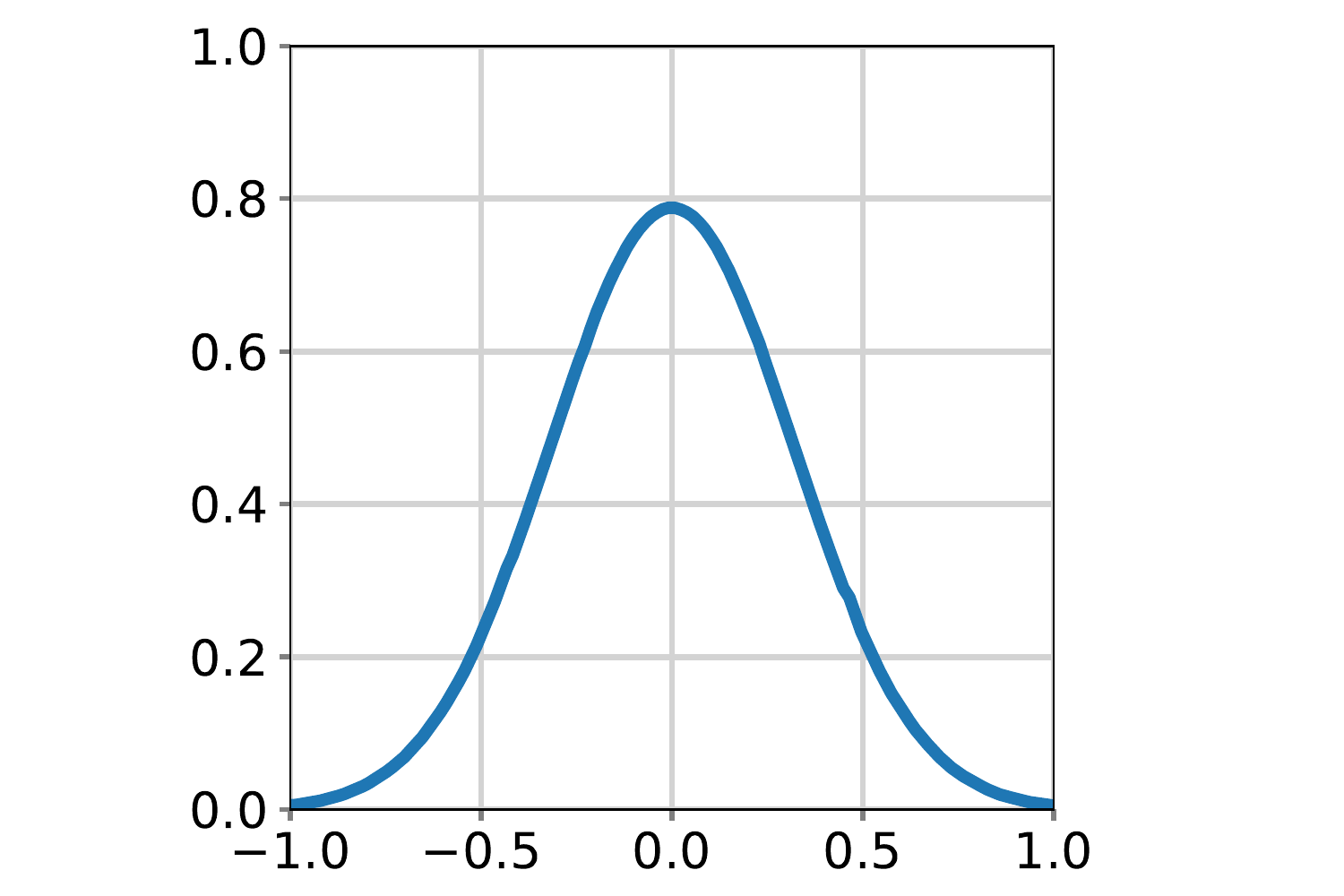}
\caption{Data-consistent}
\label{fig:regular-d}
\end{subfigure}
\caption{Reconstructions of a sample from the regular test set. In the bottom the horizontal central slice is shown. Both U-Net and data-consistent network provide an almost perfect reconstruction.}
\label{fig:regular}
\end{figure}

In Figure \ref{fig:modified} the results for one sample are shown for the modified test set, which contains smaller Gaussians with a slightly lower intensity. Both U-Net and the data-consistent network are not perfectly able to fill in the missing information in the small Gaussians. This can be expected, since Gaussians of this size were not included in the training set. However, the data-consistent network does not deform the Gaussian at the location where it is not saturated, while U-Net does this slightly; for instance around $-0.5$ in the slice plot. This behaviour is also reflected in the PSNR and SSIM values in Table \ref{tab:quality_gaussians}. Interestingly, the pseudo-inverse behaves very well if we just look at the values in the table, because the saturation mapping did not destroy a lot of the information in the Gaussian. Visual results of three more samples in the modified test set are shown in appendix \ref{app:Gaussians}. It can be seen that U-Net tends to widen the Gaussians, since it was trained on Gaussians in the training set that were wider. Although the modified test set shows a very specific modification, it illustrates that a data-consistent network is beneficial over using an arbitrary neural network: by making use of the information that we have from the mapping $\Fo$, we obtain generalization capacity.

\begin{figure}[!ht]
\centering
\begin{subfigure}[t]{0.21\textwidth}
\centering
\captionsetup{justification=centering}
\includegraphics[width=0.8\textwidth]{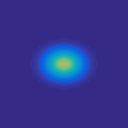}
\includegraphics[width=\textwidth, trim = {61 0 61 0}, clip]{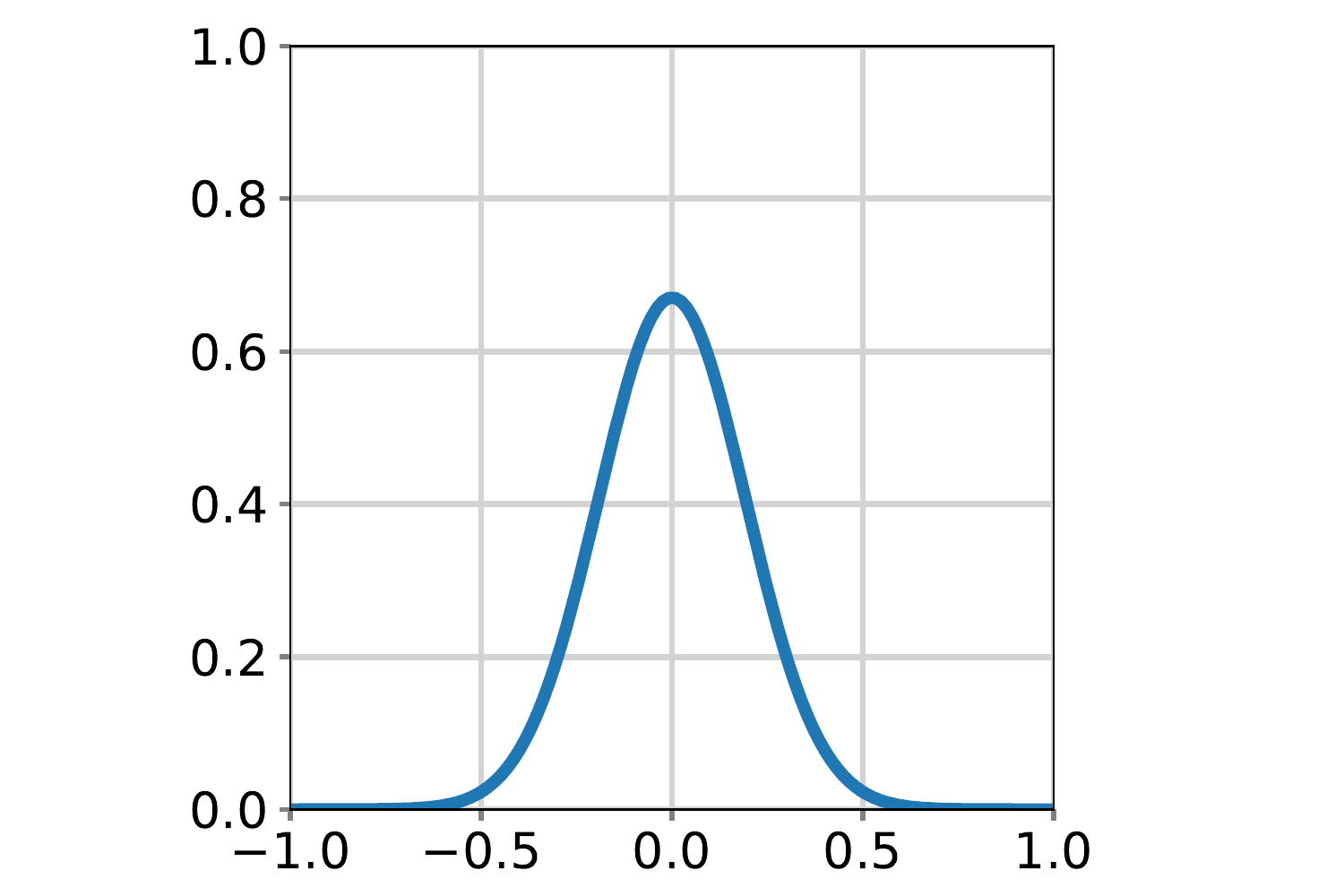}
\caption{Ground truth}
\label{fig:modified-x}
\end{subfigure}
\begin{subfigure}[t]{0.21\textwidth}
\centering
\captionsetup{justification=centering}
\includegraphics[width=0.8\textwidth]{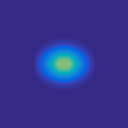}
\includegraphics[width=\textwidth, trim = {61 0 61 0}, clip]{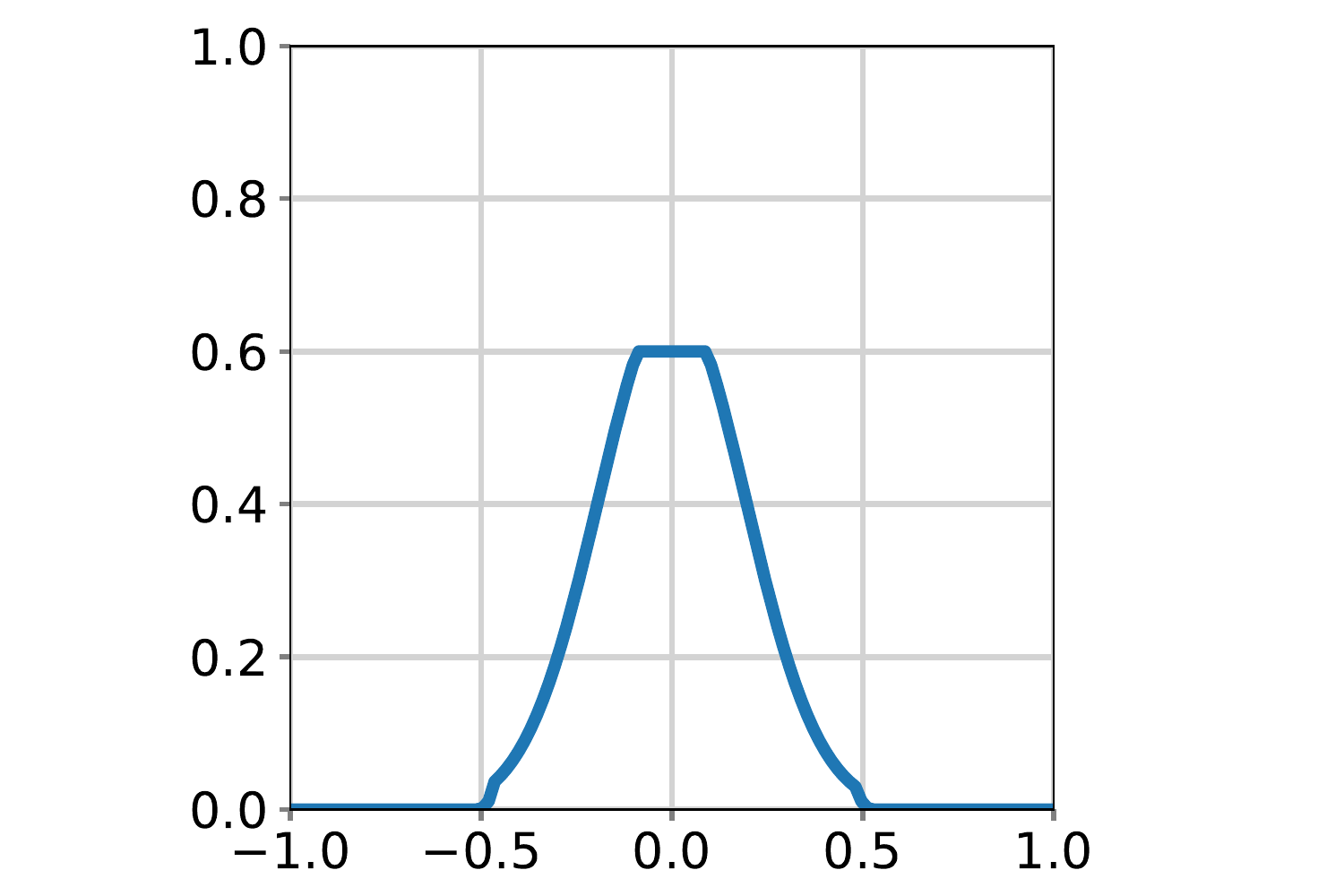}
\caption{Saturated}
\label{fig:modified-y}
\end{subfigure}
\begin{subfigure}[t]{0.21\textwidth}
\centering
\captionsetup{justification=centering}
\includegraphics[width=0.8\textwidth]{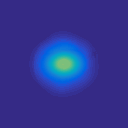}
\includegraphics[width=\textwidth, trim = {61 0 61 0}, clip]{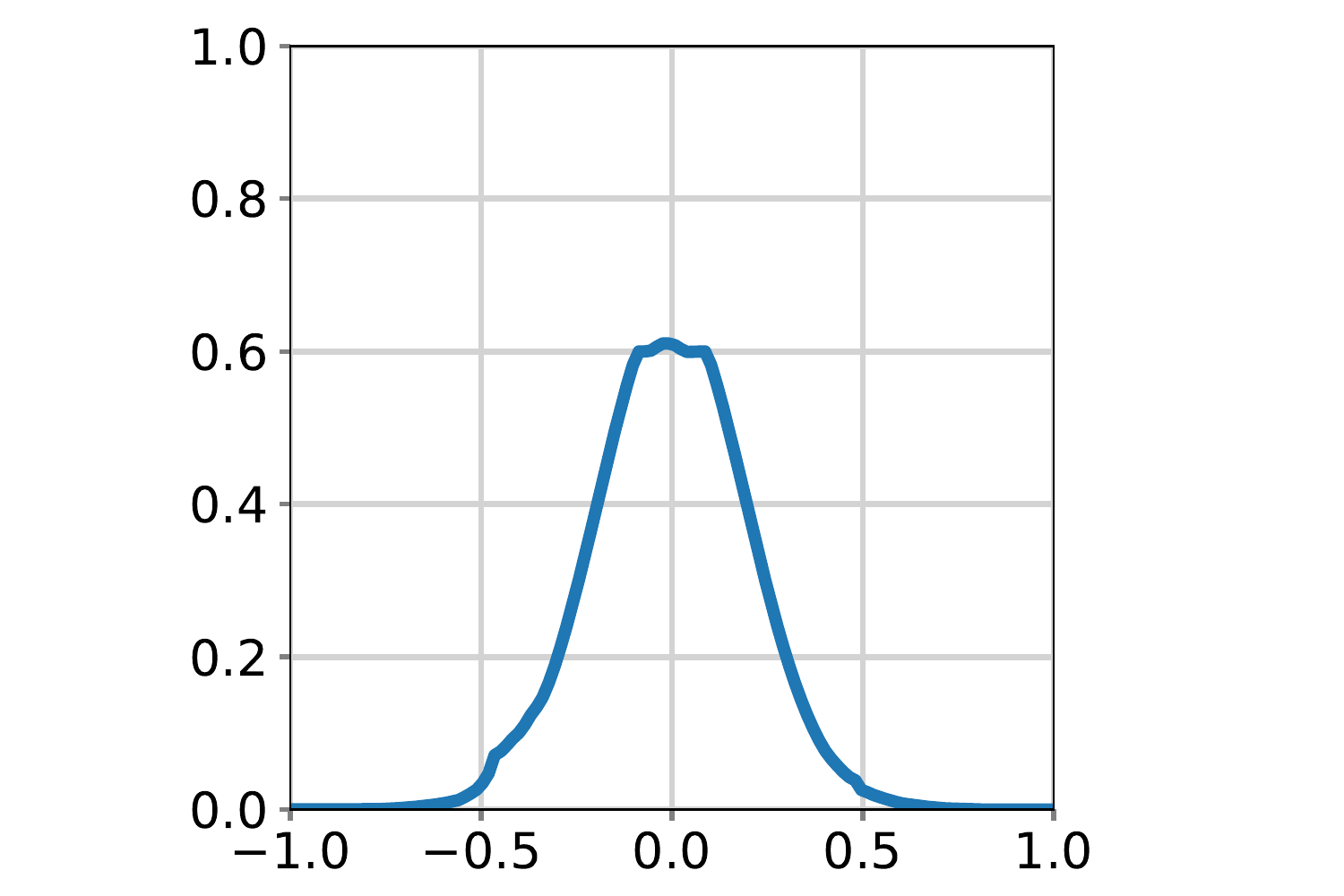}
\caption{U-Net}
\label{fig:modified-u}
\end{subfigure}
\begin{subfigure}[t]{0.21\textwidth}
\centering
\captionsetup{justification=centering}
\includegraphics[width=0.8\textwidth]{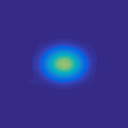}
\includegraphics[width=\textwidth, trim = {61 0 61 0}, clip]{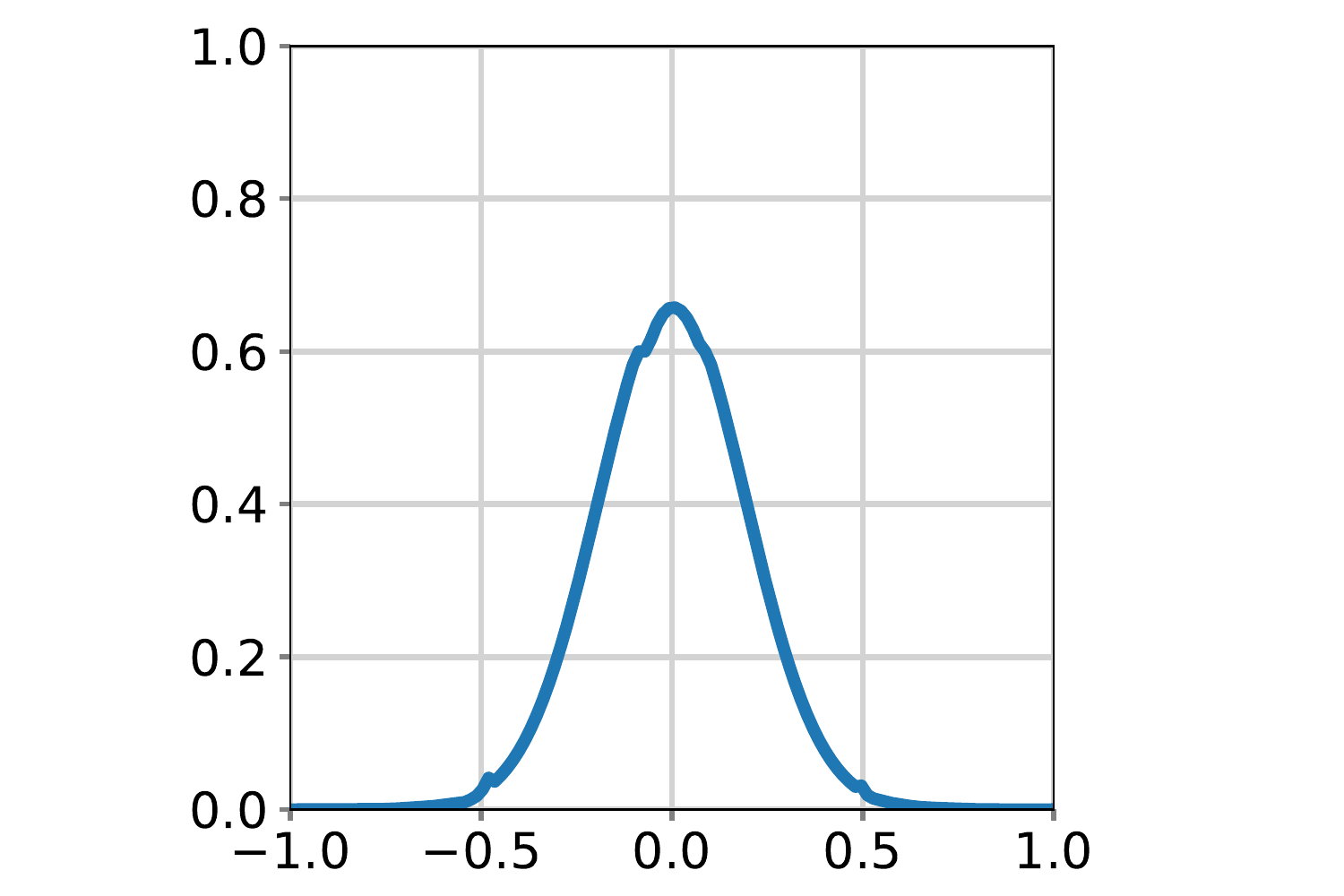}
\caption{Data-consistent}
\label{fig:modified-d}
\end{subfigure}
\caption{Reconstructions of a sample from the modified test set. In the bottom the horizontal central slice is shown. Data consistency makes sure that intensity is only changed above the saturation level.}
\label{fig:modified}
\end{figure}

\begin{table}[!ht]
\vspace{-2mm}
\begin{center}
\scalebox{0.94}[1.0]
{\begin{tabular}{*8{l}} 
\toprule

			 &				&	PSNR		&			 	&				&	SSIM		&			 	\\
			 \cmidrule(lr){2-4}					 			\cmidrule(l){5-7}					 
		 	 & Pseudo-		& 		 		& Data-		 	& Pseudo-		& 		 		& Data-		 	\\
			 & inverse		& U-Net			& consistent 	& inverse		& U-Net			& consistent 	\\
\midrule
Regular	set	 & $24.2 \pm2.2$ & $60.6 \pm2.1$ & $66.7 \pm1.6$ & $0.56 \pm 0.08$ & $1.00 \pm0.00$ & $1.00 \pm0.00$ \\
Modified set & $48.0 \pm7.8$ & $36.9 \pm2.9$ & $48.0 \pm4.4$ & $0.99 \pm 0.01$ & $0.92 \pm0.03$ & $0.97 \pm0.01$ \\
\bottomrule
\end{tabular}} 
\end{center}
\caption {Comparison of PSNR and SSIM for all reconstruction methods.}
\label{tab:quality_gaussians}
\end{table}

\subsection{Saturation of Radon transformed human chest images}
For one selected sample in the regular test set and one in the modified test set, all reconstructions are shown in Figures \ref{fig:regular-} and \ref{fig:modified-}. These specific samples were selected because their PSNR values for the U-Nets and the data-consistent network show a similar relation to each other as the average PSNR values of the whole test set (c.f. Table \ref{tab:quality_Radon}). 

\begin{figure}[!ht]
\centering
\begin{subfigure}[t]{0.193\textwidth}
\centering
\captionsetup{justification=centering}
\includegraphics[width=\textwidth, trim = {68 0 92 0}, clip]{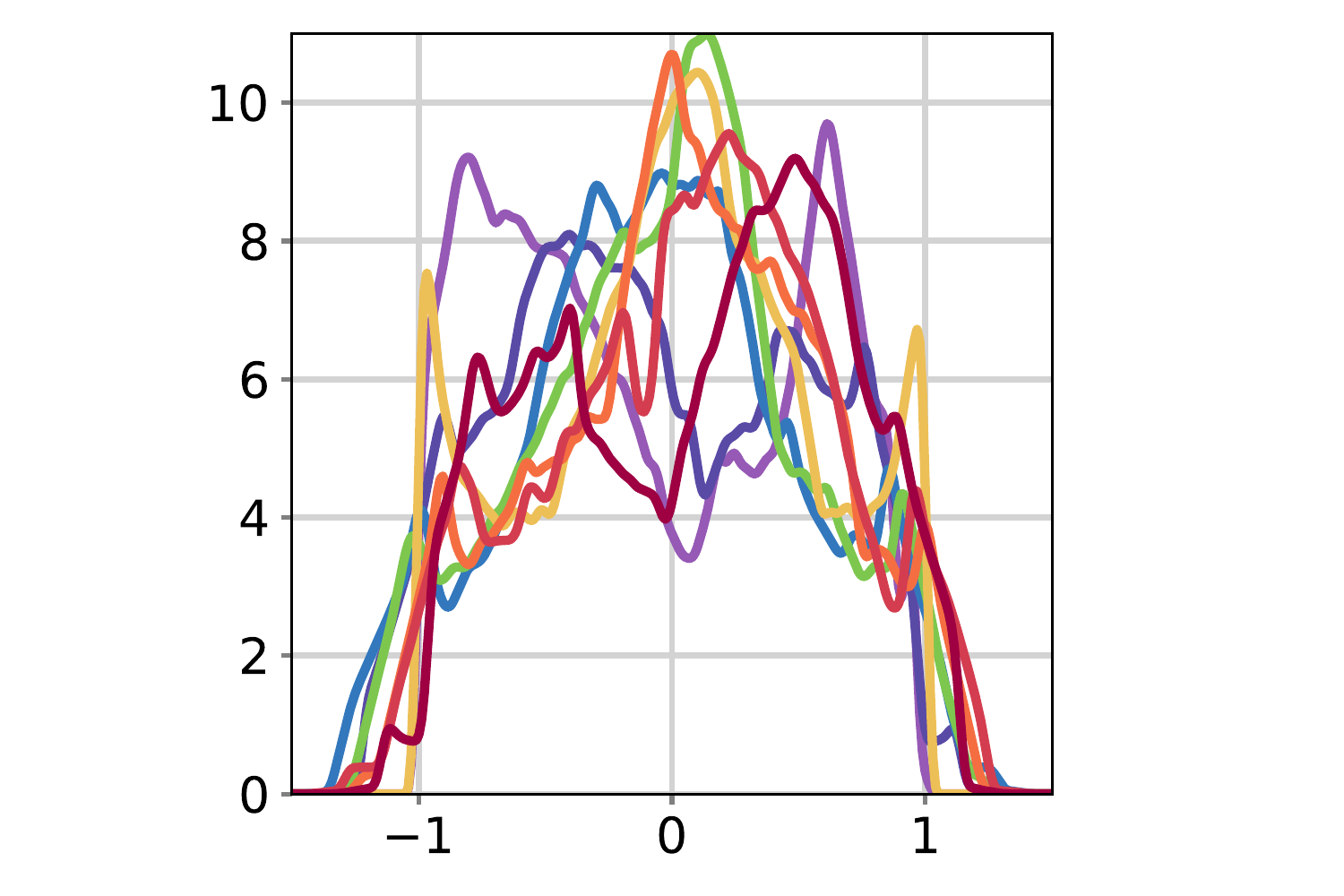}
\includegraphics[width=\textwidth, trim = {-20 0 0 0}, clip]{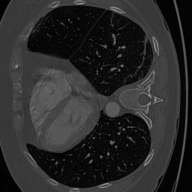}
\caption{\scalebox{0.95}[1.0]{Ground truth}}
\label{fig:regular-GT}
\end{subfigure}
\begin{subfigure}[t]{0.193\textwidth}
\centering
\captionsetup{justification=centering}
\includegraphics[width=\textwidth, trim = {68 0 92 0}, clip]{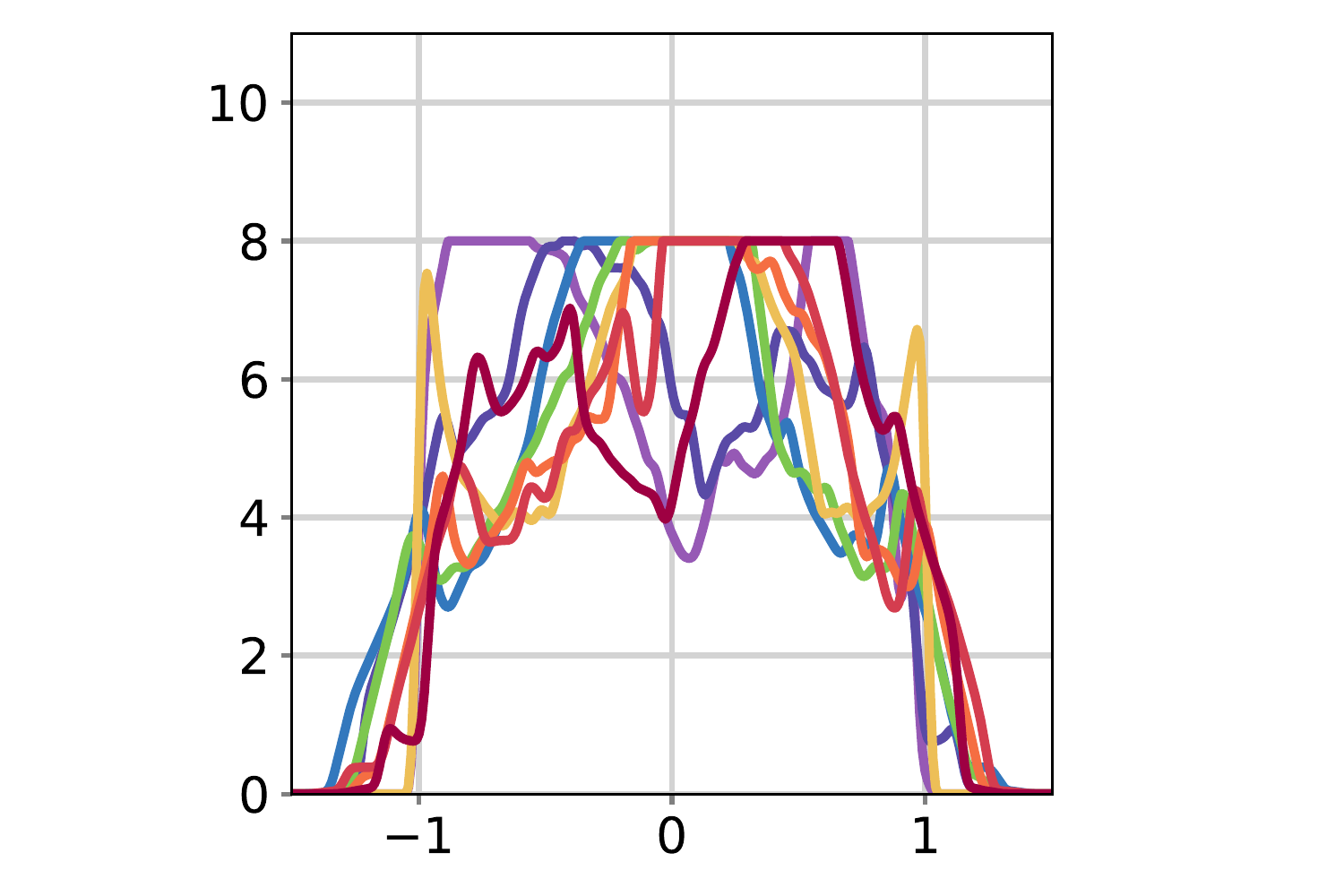}
\includegraphics[width=\textwidth, trim = {-20 0 0 0}, clip]{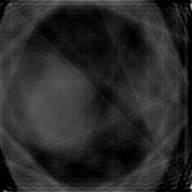}
\caption{\scalebox{0.95}[1.0]{Pseudo-inverse}\\{PSNR $=22.80$}}
\label{fig:regular-PI}
\end{subfigure}
\begin{subfigure}[t]{0.193\textwidth}
\centering
\captionsetup{justification=centering}
\includegraphics[width=\textwidth, trim = {68 0 92 0}, clip]{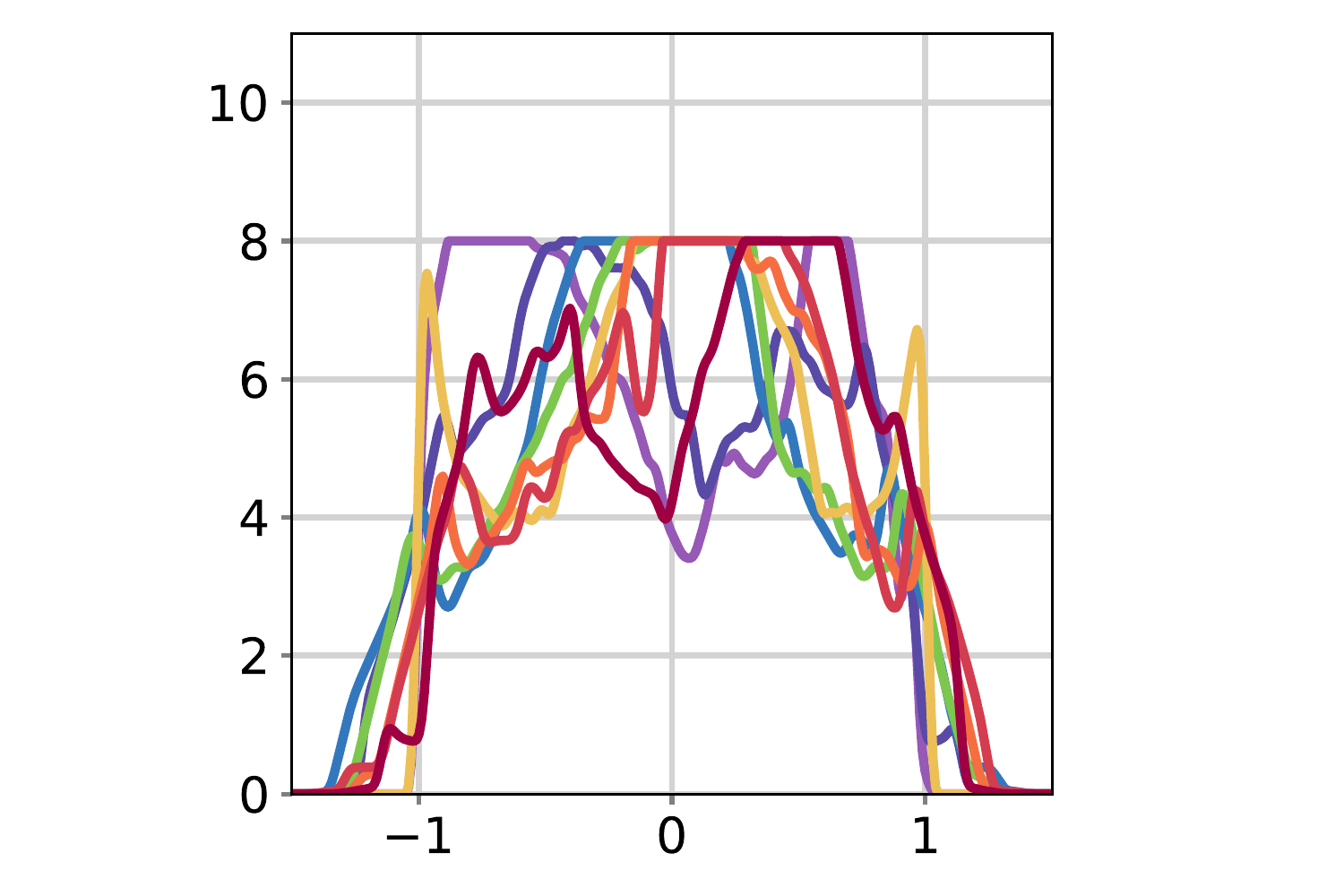}
\includegraphics[width=\textwidth, trim = {-20 0 0 0}, clip]{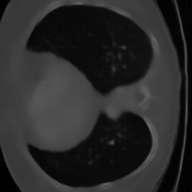}
\caption{\scalebox{0.95}[1.0]{One U-Net}\\{PSNR $=27.31$}}
\label{fig:regular-U1}
\end{subfigure}
\begin{subfigure}[t]{0.193\textwidth}
\centering
\captionsetup{justification=centering}
\includegraphics[width=\textwidth, trim = {68 0 92 0}, clip]{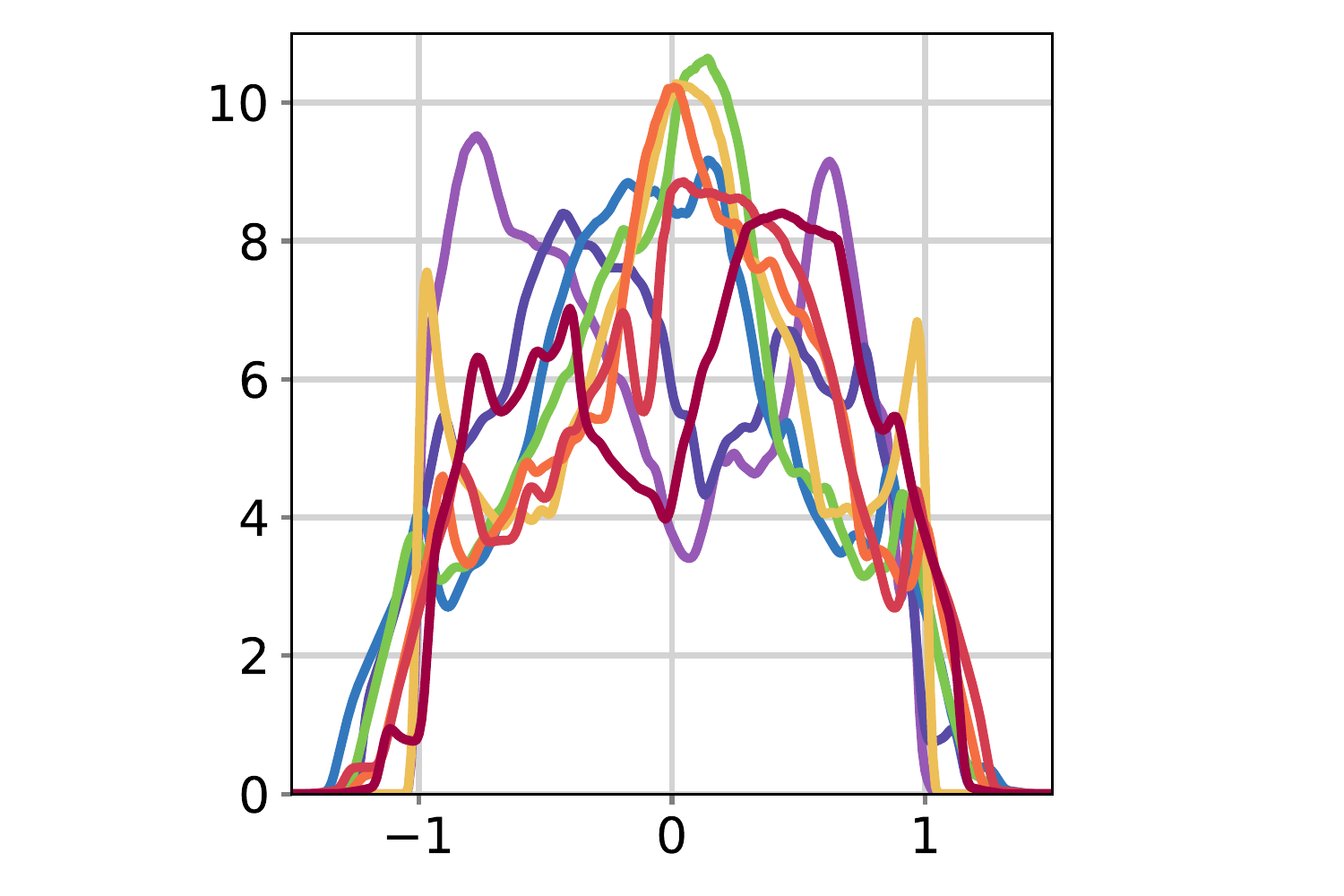}
\includegraphics[width=\textwidth, trim = {-20 0 0 0}, clip]{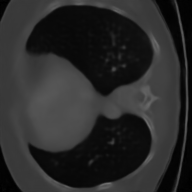}
\caption{\scalebox{0.95}[1.0]{Two U-Nets}\\{PSNR $=27.70$}}
\label{fig:regular-U2}
\end{subfigure}
\begin{subfigure}[t]{0.193\textwidth}
\centering
\captionsetup{justification=centering}
\includegraphics[width=\textwidth, trim = {68 0 92 0}, clip]{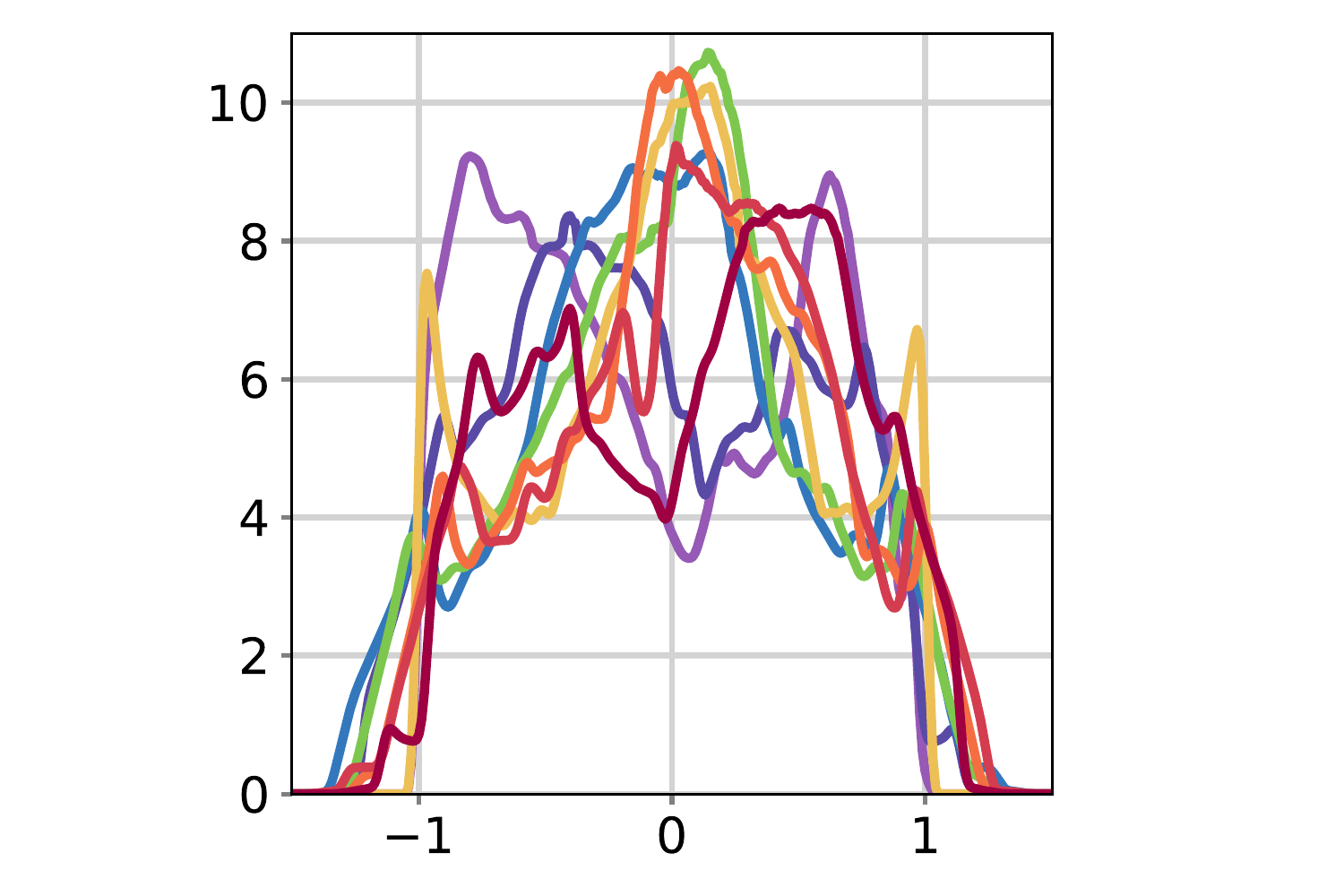}
\includegraphics[width=\textwidth, trim = {-20 0 0 0}, clip]{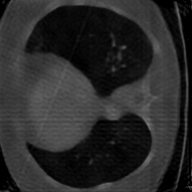}
\caption{\scalebox{0.95}[1.0]{Data-consistent}\\{PSNR $=27.01$}}
\label{fig:regular-DC}
\end{subfigure}
\caption{Reconstructions of a typical sample from the regular test set. Top: reconstructed sinograms with all 8 angles in different colors. Bottom: reconstructed images.}
\label{fig:regular-}
\end{figure}

In the top of Figure \ref{fig:regular-}, the inputs of the right inverse $\Go_0^{(1)}$ are shown. This corresponds to the saturated sinograms in case no or only one neural network is trained, and this corresponds to the output of the neural network in the sinogram domain in case two neural networks are trained. The sinogram signals are plotted in a different color for each angle. The data-consistent network does not change the values of the sinogram that are below the saturation level ($M=8$). Interestingly, the U-Net in the sinogram domain has learned not to do this as well to a large extent: some values just below the saturation level are changed (for instance the purple line around $-0.5$), but values much lower are not changed at all. In the bottom of Figure \ref{fig:regular-}, all reconstructions for this sample are shown. It can be seen that although PSNR values are similar, there is a clear visual difference: the data-consistent network shows more artefacts that are typical for limited-angle Radon reconstructions, while the standard U-Nets provide over-smoothed reconstructions. In other words, the data-consistent reconstruction does not smooth out potential details, while the other networks do. However, all networks perform rather similarly, since there are no clear structures that can be seen in one of the reconstructions and not in the others. This is also reflected in Table \ref{tab:quality_Radon}. Moreover, since only 8 angles were used in the Radon transform, only larger structures can be reconstructed. Some extreme samples, for which the PSNR value of the data-consistent solution is high, similar or low compared to the U-Net solutions, are shown in Appendix \ref{app:regular}.

\begin{figure}[!ht]
\centering
\begin{subfigure}[t]{0.193\textwidth}
\centering
\captionsetup{justification=centering}
\includegraphics[width=\textwidth, trim = {68 0 92 0}, clip]{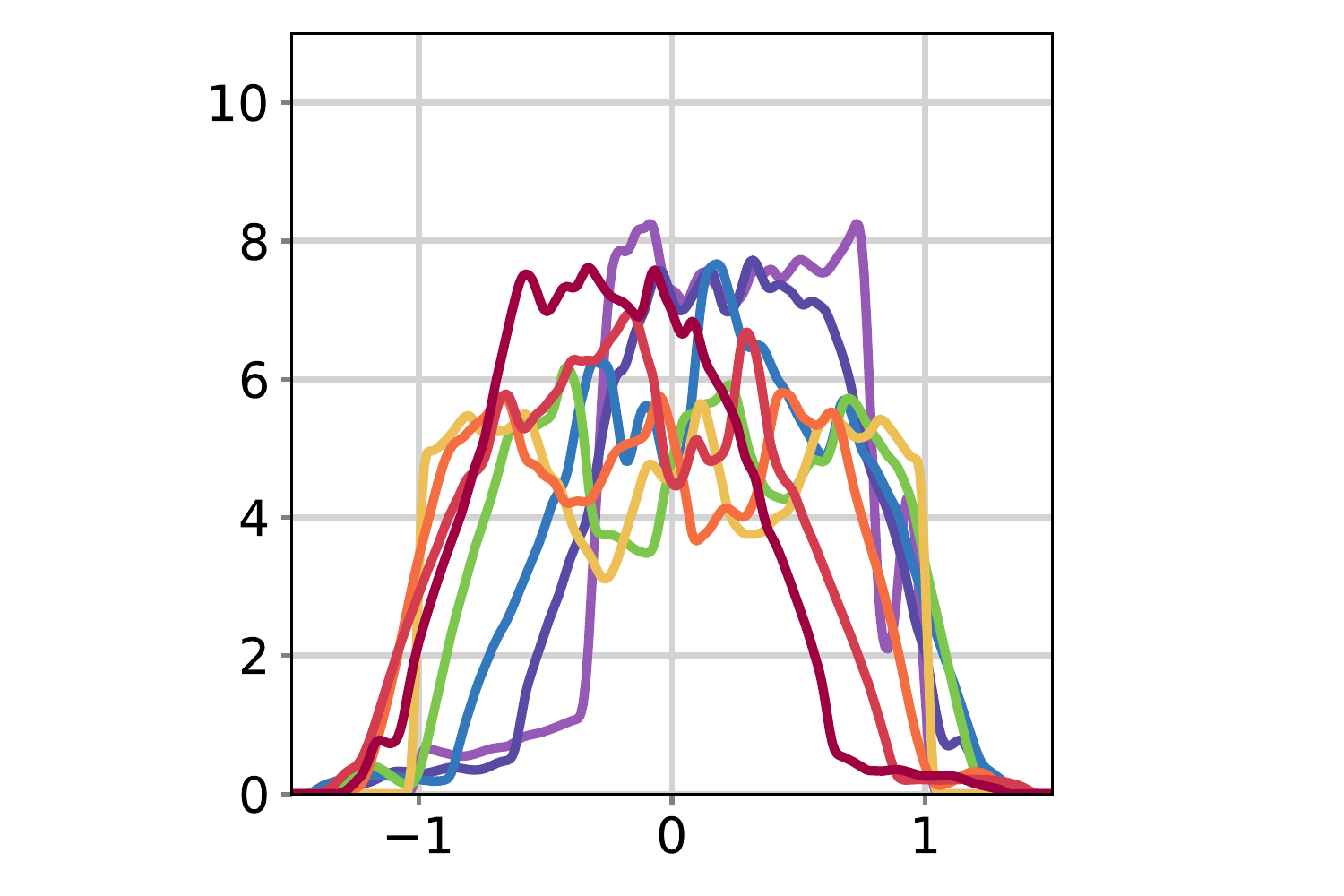}
\includegraphics[width=\textwidth, trim = {-20 0 0 0}, clip]{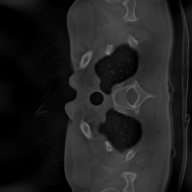}
\caption{\scalebox{0.95}[1.0]{Ground truth}}
\label{fig:modified-GT}
\end{subfigure}
\begin{subfigure}[t]{0.193\textwidth}
\centering
\captionsetup{justification=centering}
\includegraphics[width=\textwidth, trim = {68 0 92 0}, clip]{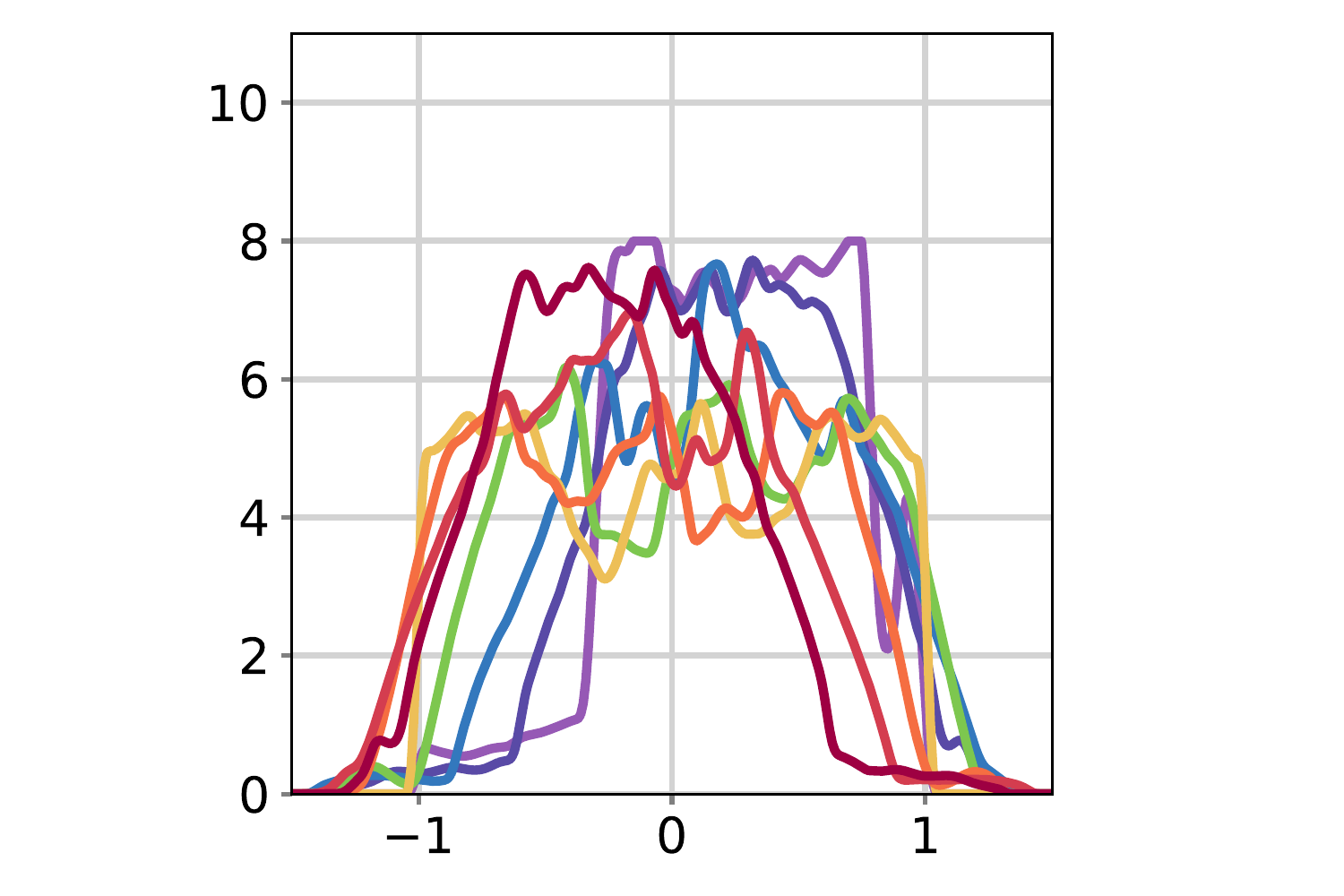}
\includegraphics[width=\textwidth, trim = {-20 0 0 0}, clip]{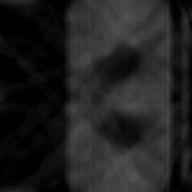}
\caption{\scalebox{0.95}[1.0]{Pseudo-inverse}\\{PSNR $=30.02$}}
\label{fig:modified-PI}
\end{subfigure}
\begin{subfigure}[t]{0.193\textwidth}
\centering
\captionsetup{justification=centering}
\includegraphics[width=\textwidth, trim = {68 0 92 0}, clip]{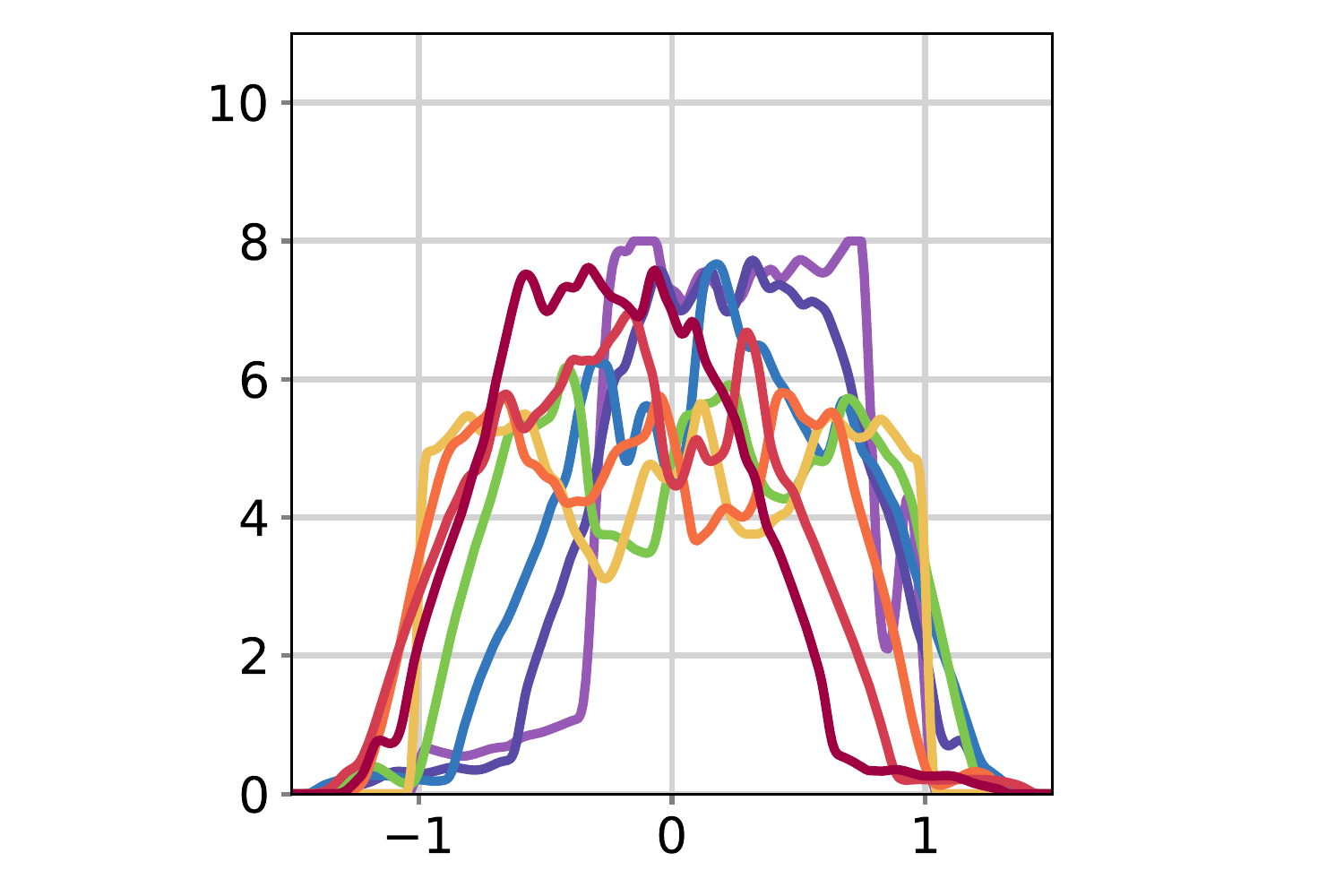}
\includegraphics[width=\textwidth, trim = {-20 0 0 0}, clip]{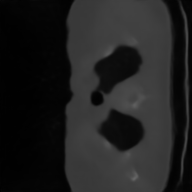}
\caption{\scalebox{0.95}[1.0]{One U-Net}\\{PSNR $=28.29$}}
\label{fig:modified-U1}
\end{subfigure}
\begin{subfigure}[t]{0.193\textwidth}
\centering
\captionsetup{justification=centering}
\includegraphics[width=\textwidth, trim = {68 0 92 0}, clip]{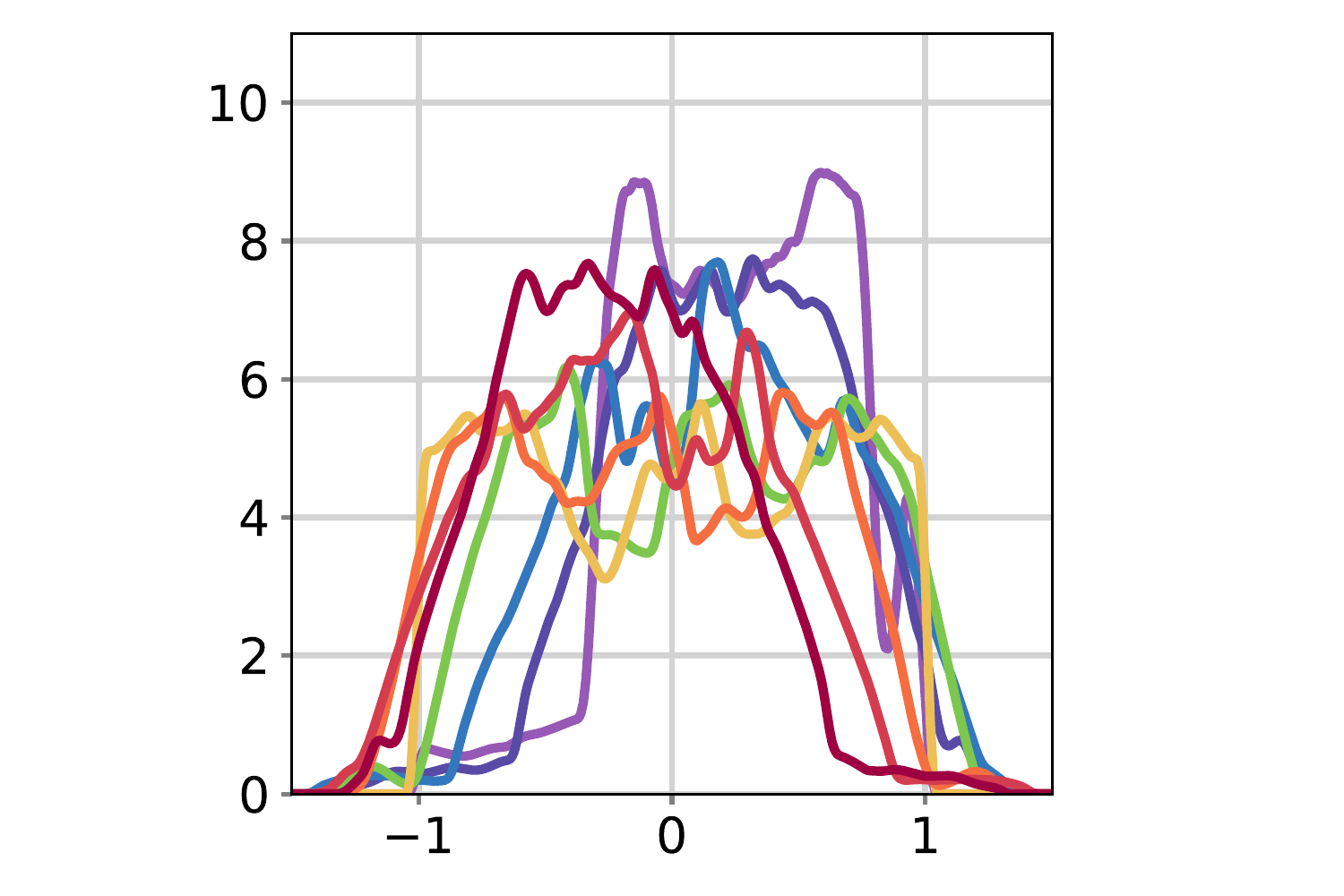}
\includegraphics[width=\textwidth, trim = {-20 0 0 0}, clip]{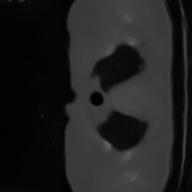}
\caption{\scalebox{0.95}[1.0]{Two U-Nets}\\{PSNR $=29.08$}}
\label{fig:modified-U2}
\end{subfigure}
\begin{subfigure}[t]{0.193\textwidth}
\centering
\captionsetup{justification=centering}
\includegraphics[width=\textwidth, trim = {68 0 92 0}, clip]{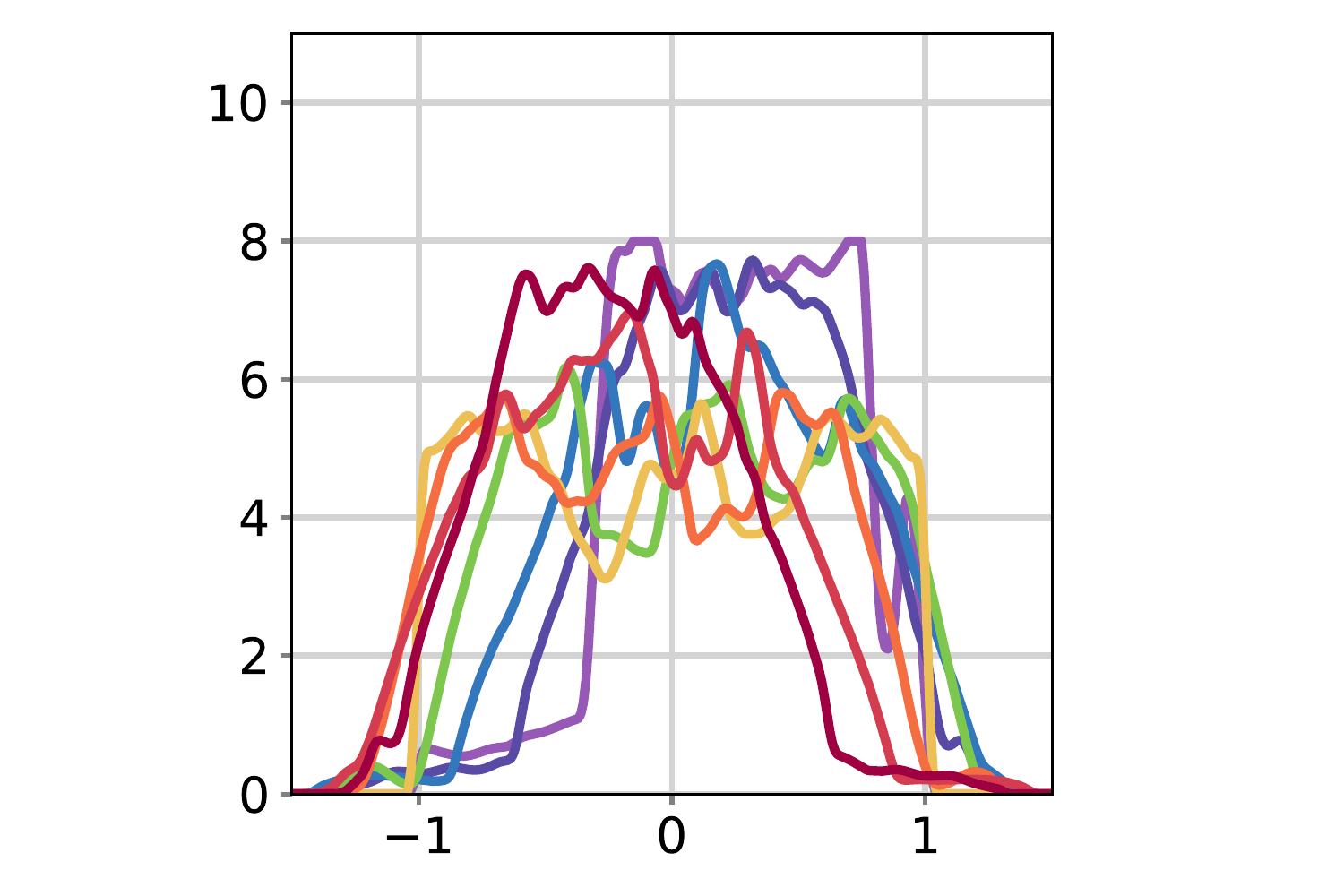}
\includegraphics[width=\textwidth, trim = {-20 0 0 0}, clip]{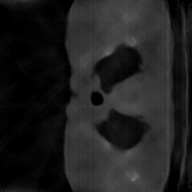}
\caption{\scalebox{0.95}[1.0]{Data-consistent}\\{PSNR $=30.83$}}
\label{fig:modified-DC}
\end{subfigure}
\caption{Reconstructions of a typical sample from the modified test set. Top: reconstructed sinograms with all 8 angles in different colors. Bottom: reconstructed images.}
\label{fig:modified-}
\end{figure}

In the top of Figure \ref{fig:modified-}, again all sinogram reconstructions are plotted, now for the modified test set. As the modified sinograms contain many values just below or around the saturation level ($M=8$), this is challenging for the regular U-Net. Indeed it can be seen that the data-consistent network does not change the saturated sinogram much, while the U-Net on the sinogram increases the purple line well beyond the saturation level. In the bottom of Figure \ref{fig:regular-}, it can be seen that the ground truth possesses some `smooth' regions, especially in the background on the left and right of the two dark inclusions. While both U-Nets create piecewise constant reconstructions that completely remove this gradient, the data-consistent network keeps these smooth regions: it generalizes better to test images that are not found in the training set by making use of the information in the operator. More extreme samples, where the PSNR value of the data-consistent solution is high, similar or low compared to the U-Net solutions, are shown in Appendix \ref{app:modified}. The same effect on smooth regions can be seen in these visualizations. In Table \ref{tab:quality_Radon}, it can be seen that for all networks the quality drops when the regular test set is replaced by the modified test set; however, this drop is only very small for the data-consistent network and is much bigger for the regular U-Nets. Note that the pseudo-inverse gives an increased PSNR for the modified test set, because the modified images were constructed to lie in the range of $\Fo$. 
\begin{table}[!ht]
\begin{center}
\scalebox{0.94}[1.0]
{\begin{tabularx}{\textwidth}{l *{3}{X} l}
\toprule			 
			 &\multicolumn{4}{c}{PSNR} \\
			 \cmidrule{2-5}
			& Pseudo-inverse	& One U-Net & Two U-Nets & Data-consistent	\\
\midrule
Regular set	 & $23.1 \pm2.3$ 	&	$30.5 \pm1.5$ 	&	$31.0 \pm1.5$ 	&	$30.1 \pm1.9$ 	\\
Modified set & $29.1 \pm1.6$ 	&	$27.5 \pm1.7$ 	&	$28.3 \pm1.4$ 	&	$29.9 \pm1.2$ 	\\
\vspace{-1.5mm}
\\
			 &\multicolumn{4}{c}{SSIM} \\
\cmidrule{2-5}
			 			& Pseudo-inverse	& One U-Net & Two U-Nets & Data-consistent	\\
\midrule
Regular set	 & $0.50 \pm0.07$ 	&	$0.82 \pm0.04$ 	&	$0.83 \pm0.04$ 	&	$0.74 \pm0.07$ 	\\
Modified set & $0.71 \pm0.07$ 	&	$0.74 \pm0.05$ 	&	$0.73 \pm0.05$ 	&	$0.75 \pm0.05$ 	\\
\vspace{-1.5mm}
\\
			 &\multicolumn{4}{c}{Data-fidelity} \\
\cmidrule{2-5}
			& Pseudo-inverse	& One U-Net & Two U-Nets & Data-consistent	\\
\midrule
Regular set & $6.1 \pm3.3$ 		&	$\hspace{1.8mm}4.8 \pm1.5$ 	&	$3.9 \pm1.1$ 	&	$0.9 \pm0.4$ 	\\
Modified set & $0.4 \pm0.2$ 	&	$11.9 \pm5.4$ 	&	$8.5 \pm2.8$ 	&	$0.6 \pm0.2$ 	\\
\bottomrule
\end{tabularx}} 
\end{center}
\caption {Comparison of PSNR, SSIM and data-fidelity for all reconstruction methods.}
\label{tab:quality_Radon}
\end{table}

Finally in Table \ref{tab:quality_Radon} we check the data-fidelity of the solutions from all networks by computing $\norm{\Fo(\tilde{\signal})-\Fo(\signal)}$, where $\tilde{\signal}$ is the solution of the respective reconstruction method and $x$ is the ground truth. Ideally, the data-fidelity should be zero for the pseudo-inverse and the data-consistent network. It can be seen that indeed the data-fidelity is much lower for these solutions than for the U-Nets, although not completely zero. This is most probably due to numerical issues (especially for the pseudo-inverse of the regular set) and due to the fact that we needed to save and load images from disc while training  because of the size of the data set. 

\section{Conclusion}
In this paper we introduced data consistent networks for nonlinear inverse problems. We presented a convergent regularization method by combining deep neural networks that converge to a data-consistent network with classical regularization methods. With the proposed data-driven regularization methods we are able to preserve convergence rates of classical methods over a transformed source set, which is adapted to some data set. This yields improved reconstructions for elements close to the training set, but at the same time data-consistent networks make use of the information from the forward mapping $\Fo$, which provides increased generalization capacity. This is particularly useful when the physics process is understood, but exact knowledge on real data is not available or when it is not possible to create a training set that is similar to the real data. We showed that on a test set similar to the training set, our approach shows reconstruction results comparable to a classical post-processing network, whereas for instances not represented in the training set, the loss of performance is much less present. This demonstrates the generalization ability of our approach.

\newpage 

\appendix
\section{Additional results for saturated Gaussians}\label{app:Gaussians}
In this appendix, three more samples from the modified test set of the saturated Gaussians are shown. It can be seen in Figures \ref{fig:modified-2} and \ref{fig:modified-3} that U-Net tends to slightly widen the Gaussian in some cases, since this was necessary for the wider Gaussians in the training set. Moreover it can be seen in Figures \ref{fig:modified-3} and \ref{fig:modified-1} that both U-Net and the data-consistent network sometimes fail to restore the top of the Gaussian adequately: the training on wider Gaussians is not directly generalised for smaller Gaussians. 

\begin{figure}[!ht]
\centering
\begin{subfigure}[t]{0.24\textwidth}
\centering
\captionsetup{justification=centering}
\includegraphics[width=0.8\textwidth]{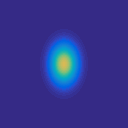}
\includegraphics[width=\textwidth, trim = {61 0 61 0}, clip]{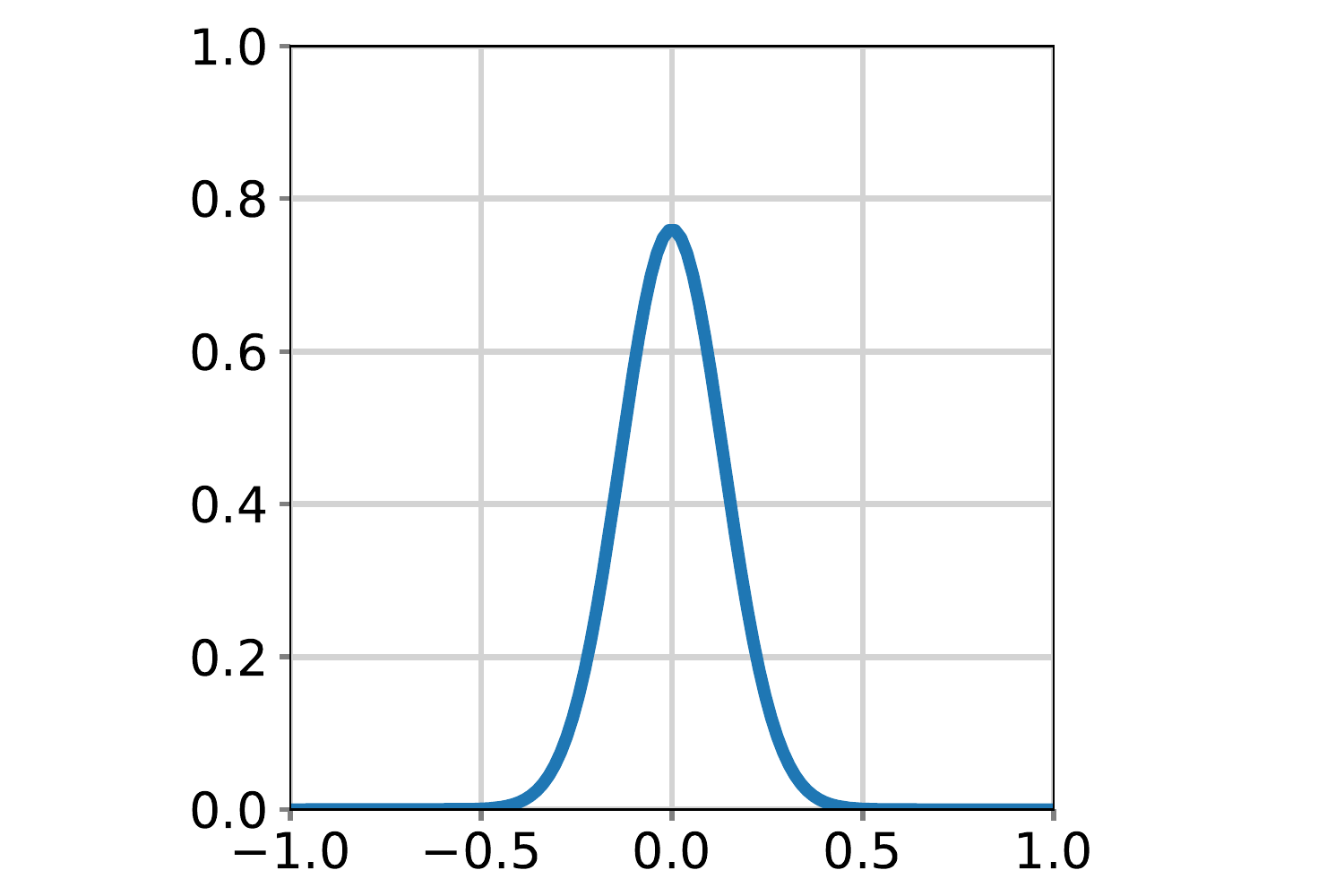}
\caption{Ground truth}
\label{fig:modified-x-2}
\end{subfigure}
\begin{subfigure}[t]{0.24\textwidth}
\centering
\captionsetup{justification=centering}
\includegraphics[width=0.8\textwidth]{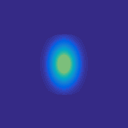}
\includegraphics[width=\textwidth, trim = {61 0 61 0}, clip]{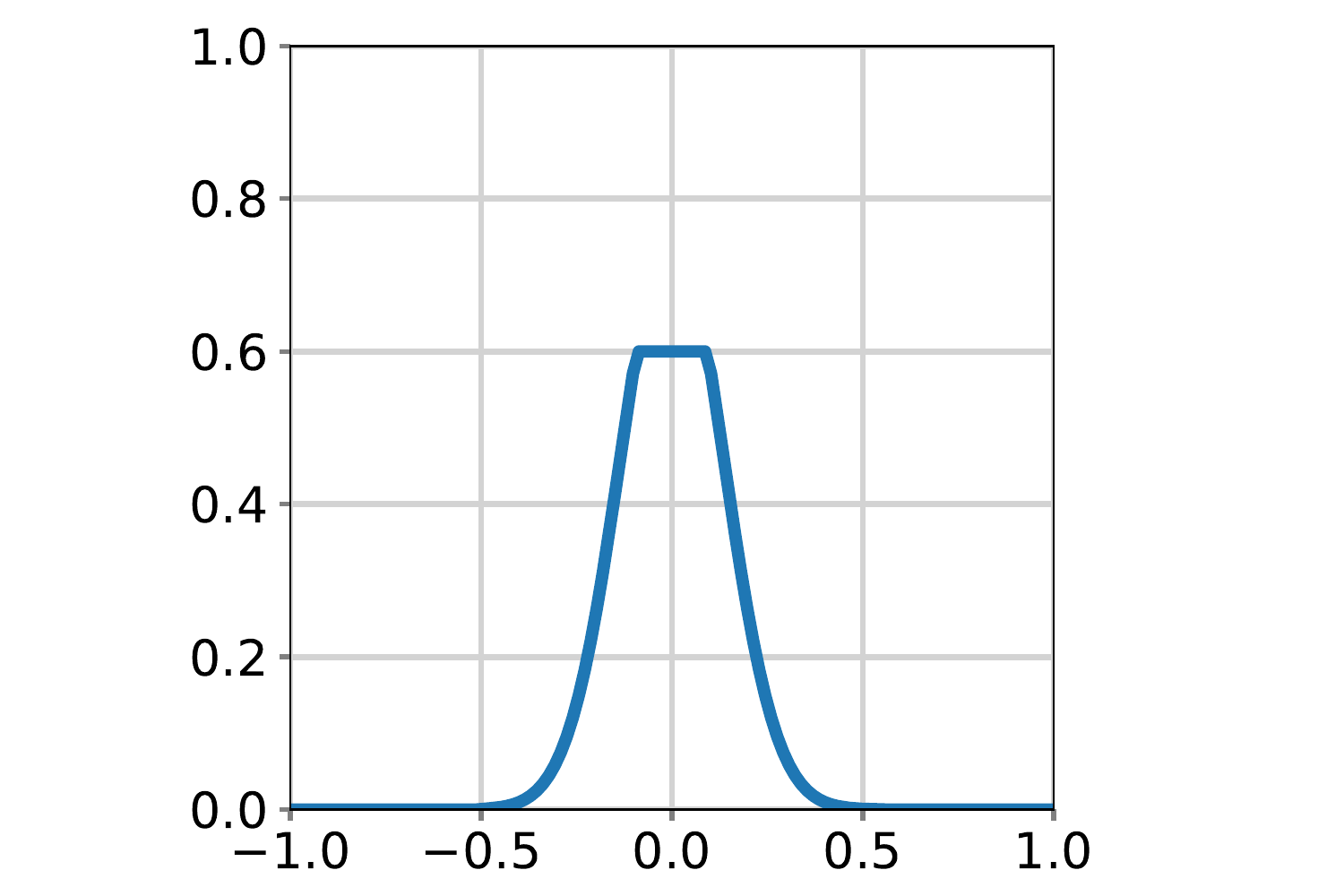}
\caption{Saturated}
\label{fig:modified-y-2}
\end{subfigure}
\begin{subfigure}[t]{0.24\textwidth}
\centering
\captionsetup{justification=centering}
\includegraphics[width=0.8\textwidth]{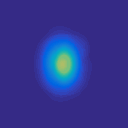}
\includegraphics[width=\textwidth, trim = {61 0 61 0}, clip]{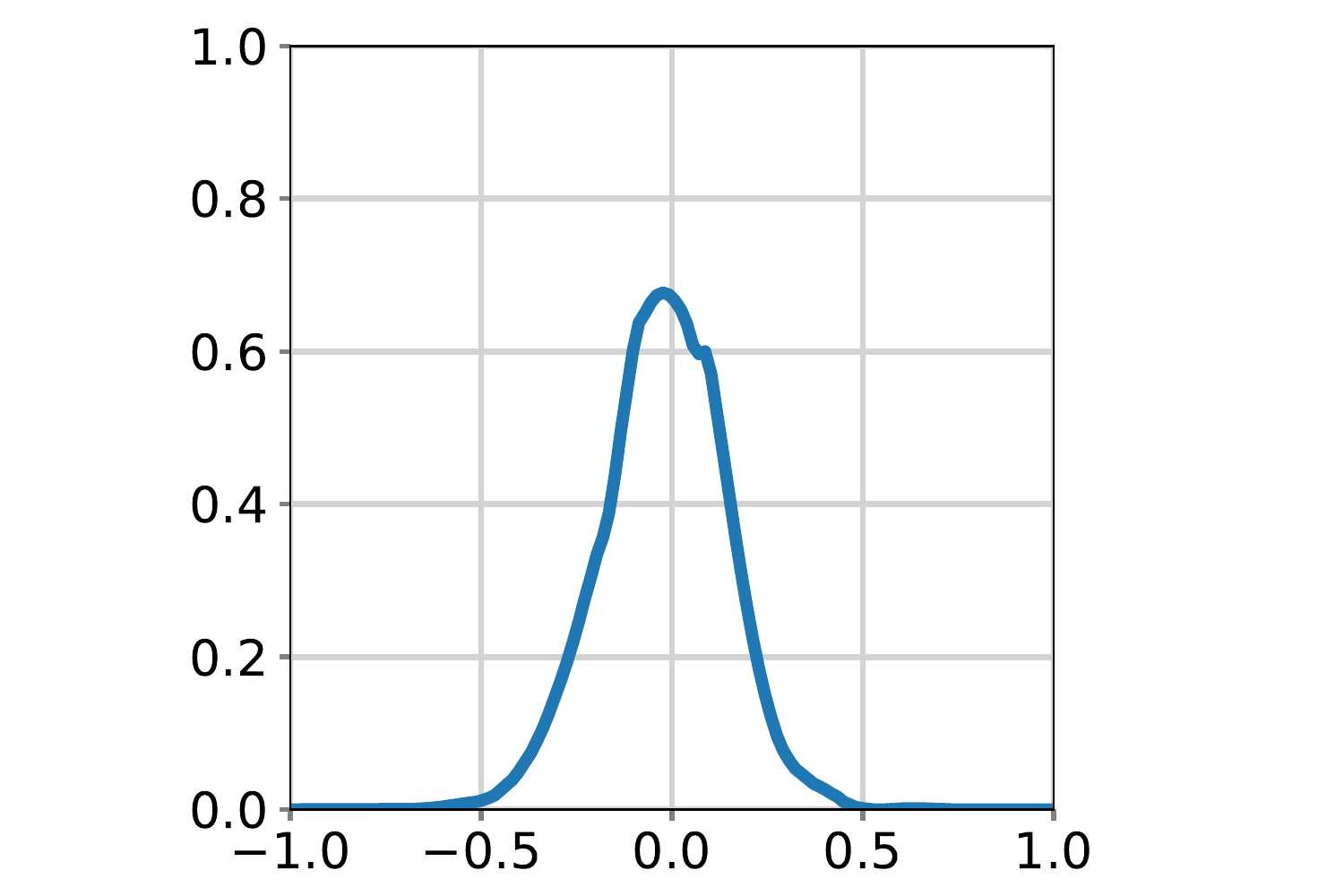}
\caption{U-Net}
\label{fig:modified-u-2}
\end{subfigure}
\begin{subfigure}[t]{0.24\textwidth}
\centering
\captionsetup{justification=centering}
\includegraphics[width=0.8\textwidth]{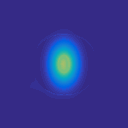}
\includegraphics[width=\textwidth, trim = {61 0 61 0}, clip]{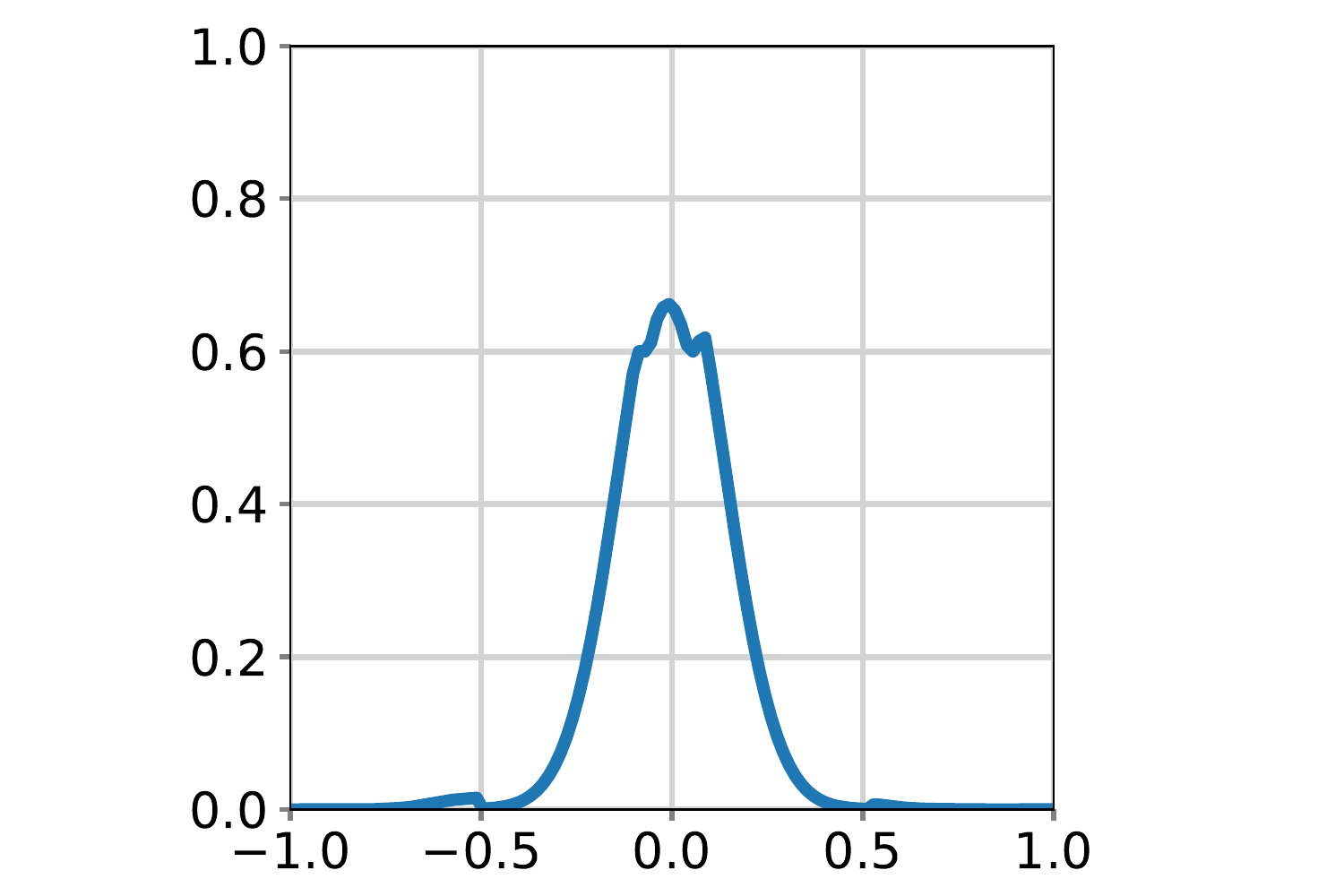}
\caption{Data-consistent}
\label{fig:modified-d-2}
\end{subfigure}
\caption{}
\label{fig:modified-2}
\end{figure}

\begin{figure}[!ht]
\centering
\begin{subfigure}[t]{0.24\textwidth}
\centering
\captionsetup{justification=centering}
\includegraphics[width=0.8\textwidth]{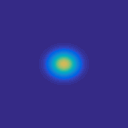}
\includegraphics[width=\textwidth, trim = {61 0 61 0}, clip]{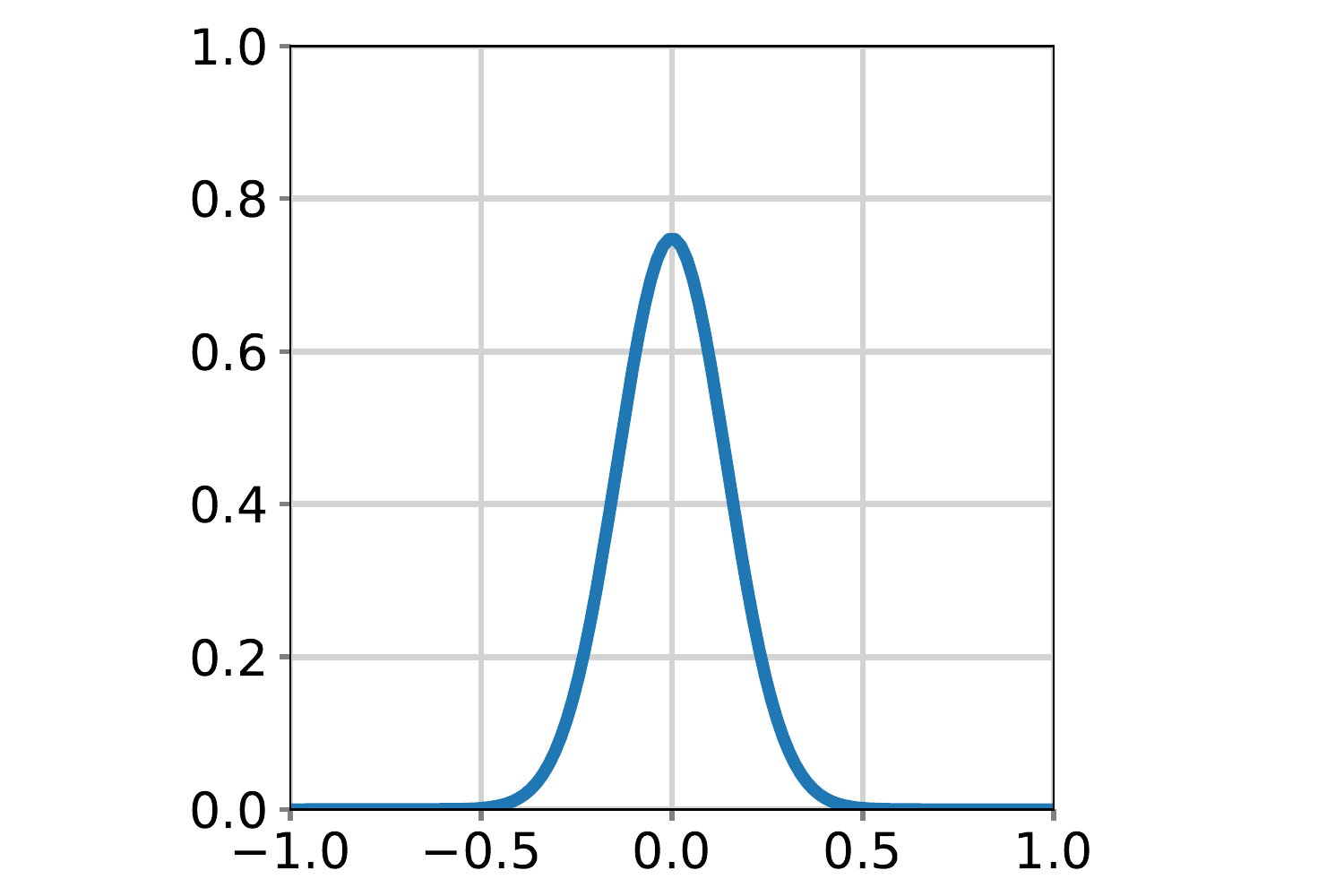}
\caption{Ground truth}
\label{fig:modified-x-3}
\end{subfigure}
\begin{subfigure}[t]{0.24\textwidth}
\centering
\captionsetup{justification=centering}
\includegraphics[width=0.8\textwidth]{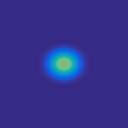}
\includegraphics[width=\textwidth, trim = {61 0 61 0}, clip]{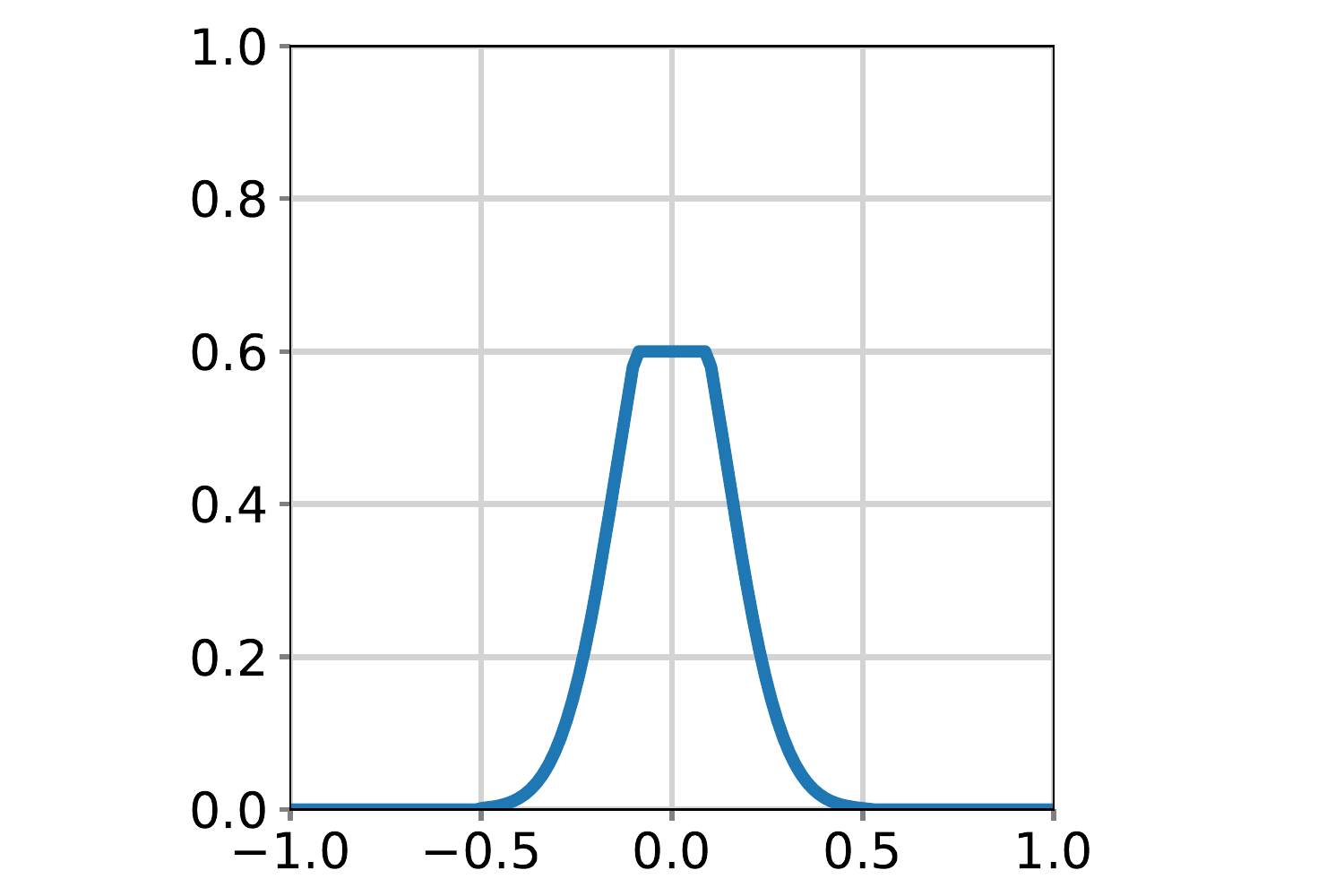}
\caption{Saturated}
\label{fig:modified-y-3}
\end{subfigure}
\begin{subfigure}[t]{0.24\textwidth}
\centering
\captionsetup{justification=centering}
\includegraphics[width=0.8\textwidth]{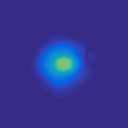}
\includegraphics[width=\textwidth, trim = {61 0 61 0}, clip]{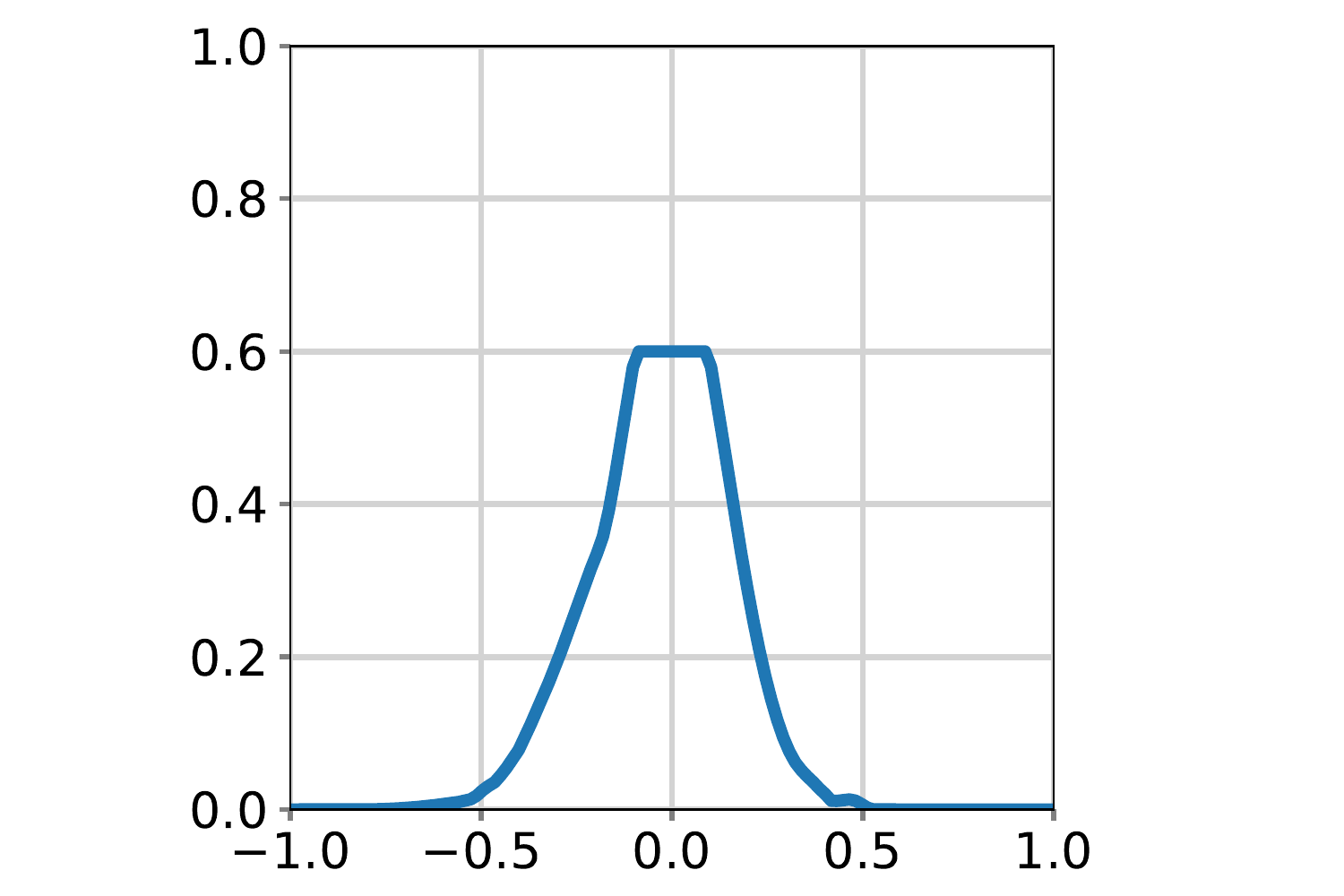}
\caption{U-Net}
\label{fig:modified-u-3}
\end{subfigure}
\begin{subfigure}[t]{0.24\textwidth}
\centering
\captionsetup{justification=centering}
\includegraphics[width=0.8\textwidth]{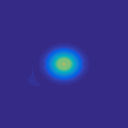}
\includegraphics[width=\textwidth, trim = {61 0 61 0}, clip]{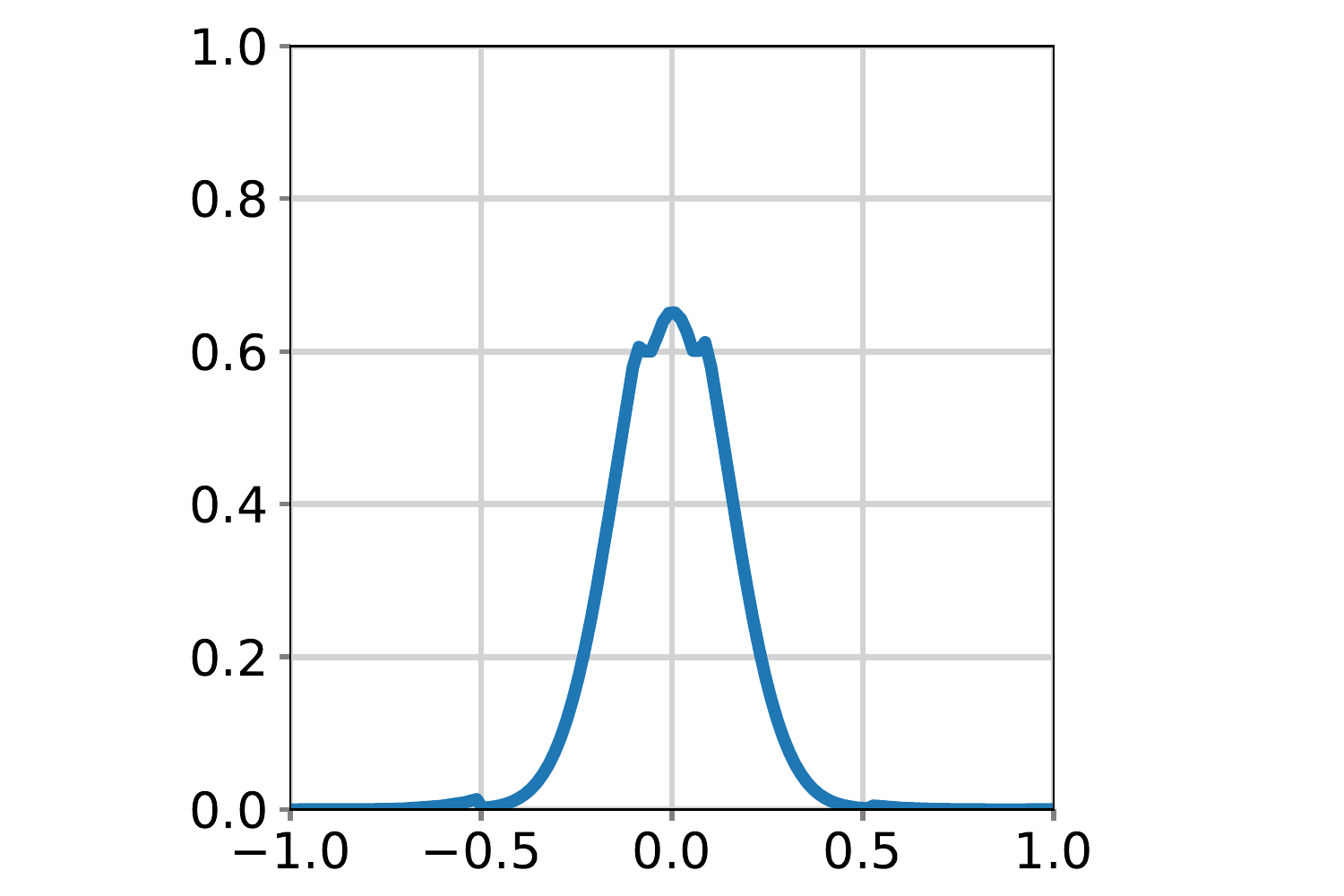}
\caption{Data-consistent}
\label{fig:modified-d-3}
\end{subfigure}
\caption{}
\label{fig:modified-3}
\end{figure}

\begin{figure}[!ht]
\centering
\begin{subfigure}[t]{0.24\textwidth}
\centering
\captionsetup{justification=centering}
\includegraphics[width=0.8\textwidth]{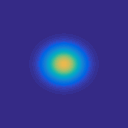}
\includegraphics[width=\textwidth, trim = {61 0 61 0}, clip]{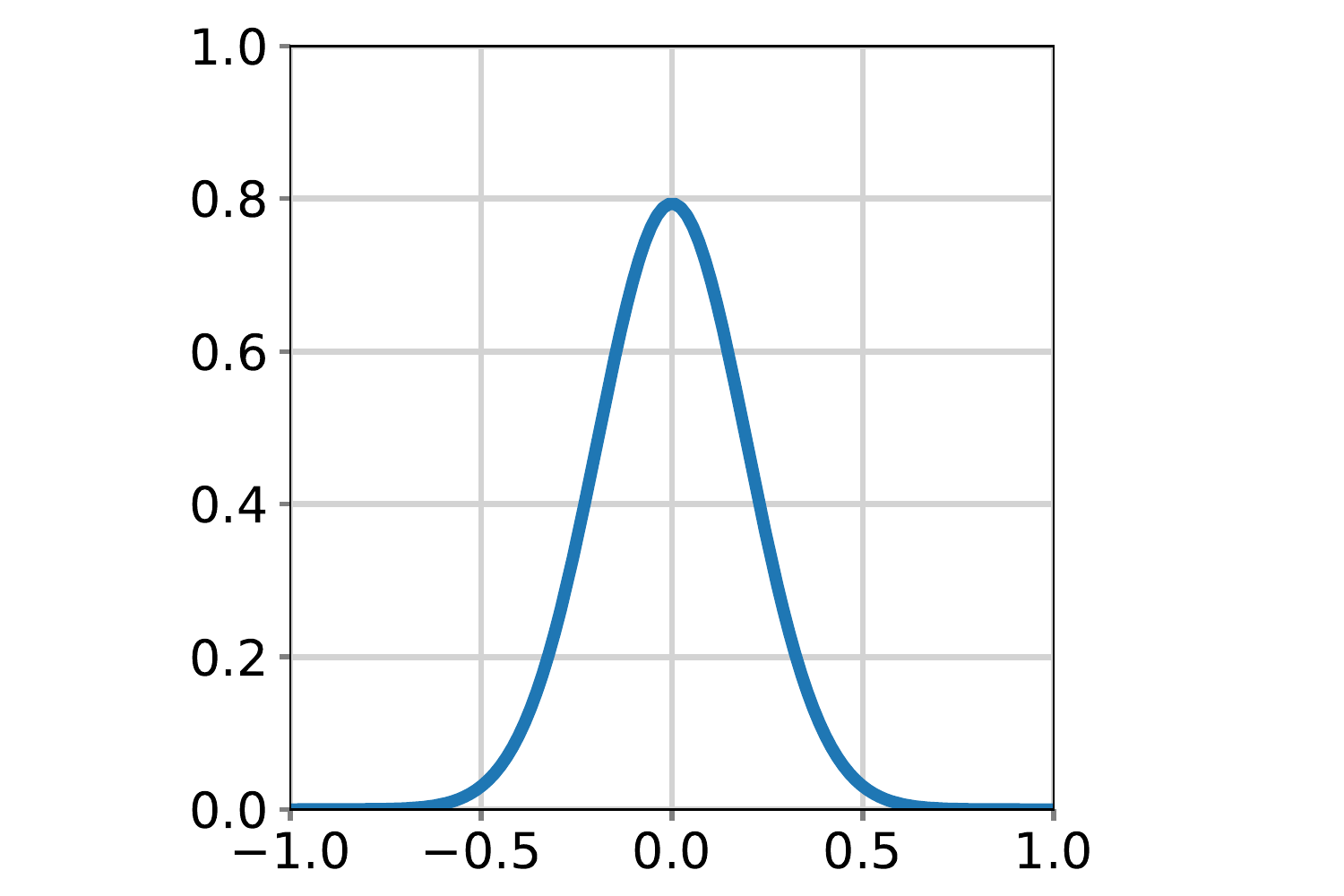}
\caption{Ground truth}
\label{fig:modified-x-1}
\end{subfigure}
\begin{subfigure}[t]{0.24\textwidth}
\centering
\captionsetup{justification=centering}
\includegraphics[width=0.8\textwidth]{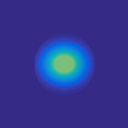}
\includegraphics[width=\textwidth, trim = {61 0 61 0}, clip]{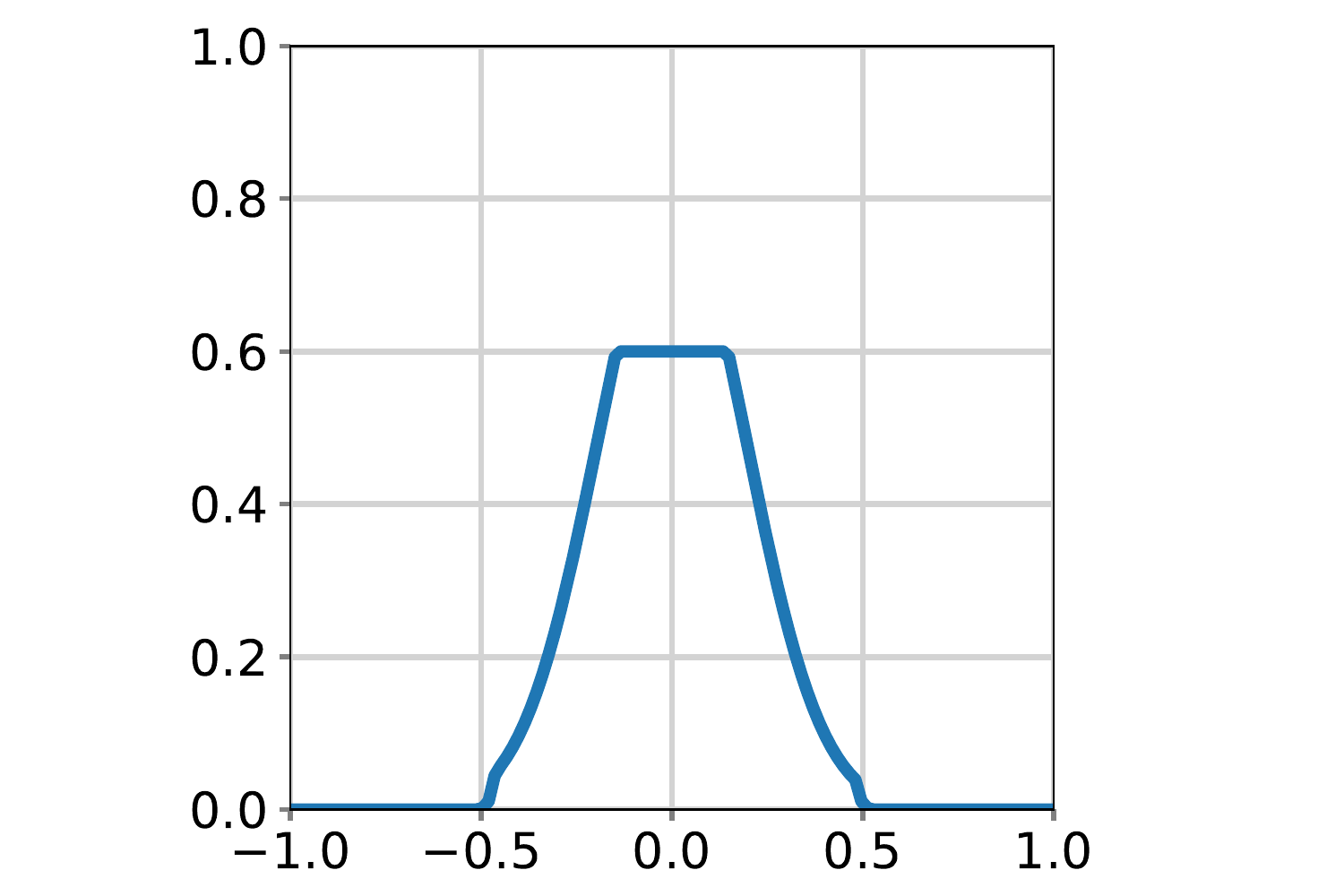}
\caption{Saturated}
\label{fig:modified-y-1}
\end{subfigure}
\begin{subfigure}[t]{0.24\textwidth}
\centering
\captionsetup{justification=centering}
\includegraphics[width=0.8\textwidth]{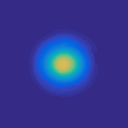}
\includegraphics[width=\textwidth, trim = {61 0 61 0}, clip]{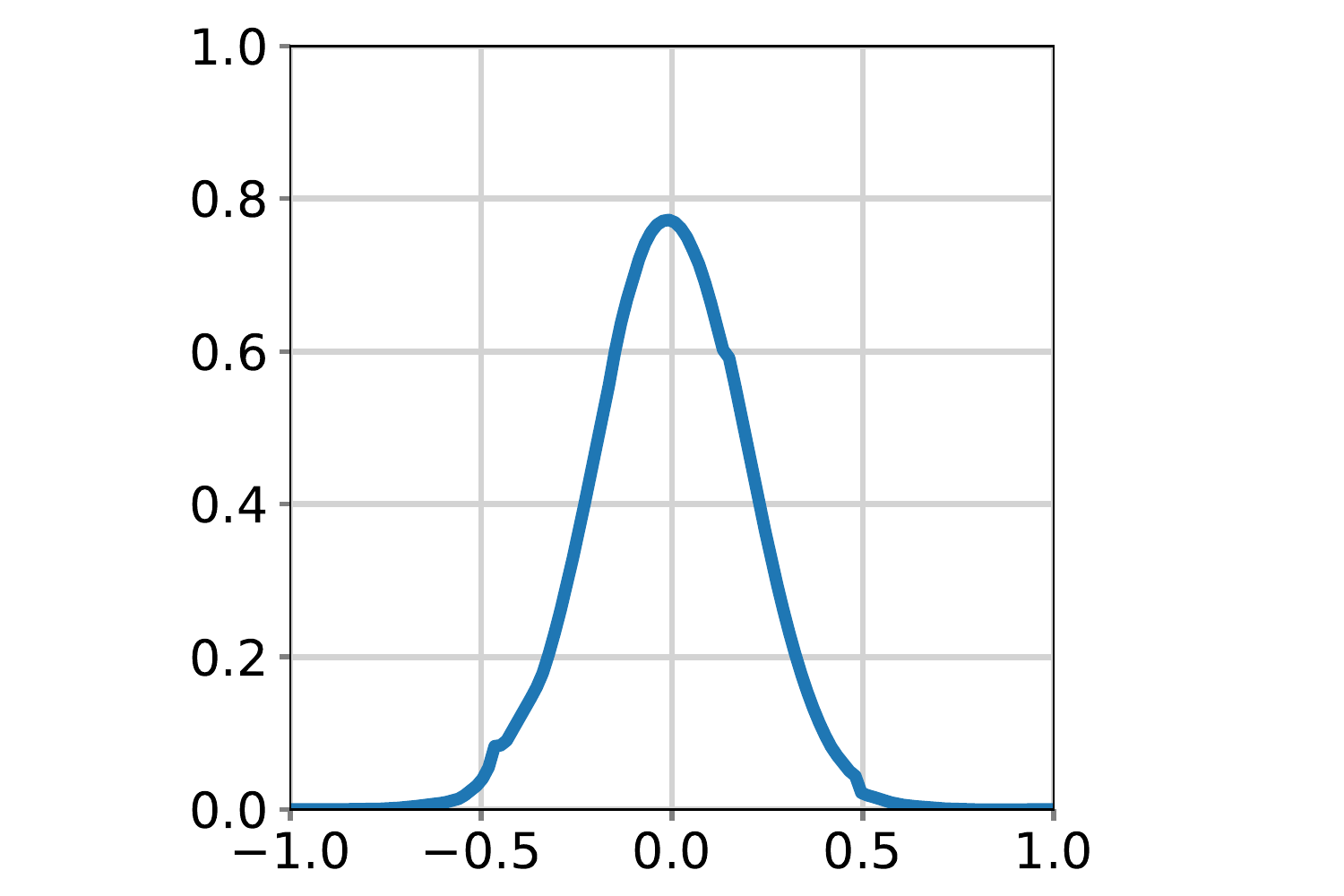}
\caption{U-Net}
\label{fig:modified-u-1}
\end{subfigure}
\begin{subfigure}[t]{0.24\textwidth}
\centering
\captionsetup{justification=centering}
\includegraphics[width=0.8\textwidth]{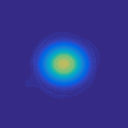}
\includegraphics[width=\textwidth, trim = {61 0 61 0}, clip]{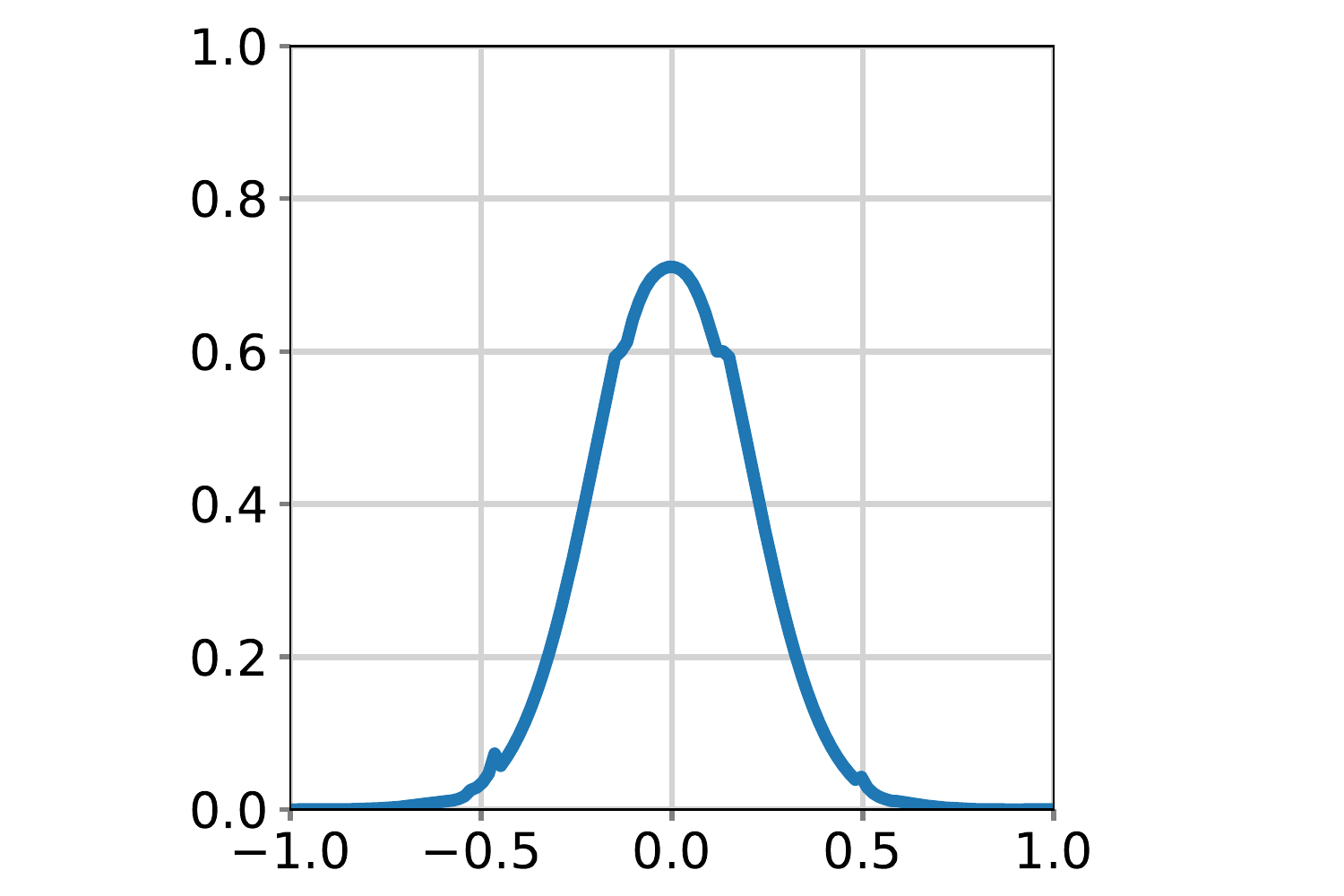}
\caption{Data-consistent}
\label{fig:modified-d-1}
\end{subfigure}
\caption{}
\label{fig:modified-1}
\end{figure}

\clearpage
\newpage
\section{Additional results for saturated Radon transform}\label{app:Radon}
In this appendix, three additional samples from the regular and the modified test set of human chest images are shown. The samples have been selected based on their PSNR values: we show the samples in the test set for which the data-consistent network yields the highest relative PSNR value (Figures \ref{fig:regular-1} and \ref{fig:modified-4}), a similar PSNR value (Figures \ref{fig:regular-2} and \ref{fig:modified-5}), and the lowest relative PSNR value (Figures \ref{fig:regular-3} and \ref{fig:modified-6}) when compared to the U-Nets.

\subsection{Regular test set}\label{app:regular}
$~$
\begin{figure}[!ht]
\centering
\begin{subfigure}[t]{0.193\textwidth}
\centering
\captionsetup{justification=centering}
\includegraphics[width=\textwidth, trim = {0 0 0 0}, clip]{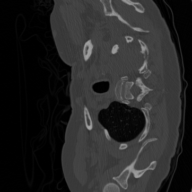}
\caption{\scalebox{0.95}[1.0]{Ground truth}}
\label{fig:regular-GT-1}
\end{subfigure}
\begin{subfigure}[t]{0.193\textwidth}
\centering
\captionsetup{justification=centering}
\includegraphics[width=\textwidth, trim = {0 0 0 0}, clip]{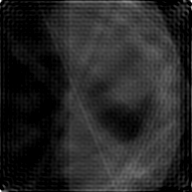}
\caption{\scalebox{0.95}[1.0]{Pseudo-inverse}\\{PSNR $=29.01$}}
\label{fig:regular-PI-1}
\end{subfigure}
\begin{subfigure}[t]{0.193\textwidth}
\centering
\captionsetup{justification=centering}
\includegraphics[width=\textwidth, trim = {0 0 0 0}, clip]{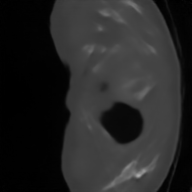}
\caption{\scalebox{0.95}[1.0]{One U-Net}\\{PSNR $=31.50$}}
\label{fig:regular-U1-1}
\end{subfigure}
\begin{subfigure}[t]{0.193\textwidth}
\centering
\captionsetup{justification=centering}
\includegraphics[width=\textwidth, trim = {0 0 0 0}, clip]{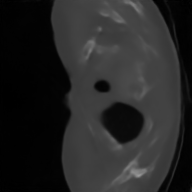}
\caption{\scalebox{0.95}[1.0]{Two U-Nets}\\{PSNR $=32.29$}}
\label{fig:regular-U2-1}
\end{subfigure}
\begin{subfigure}[t]{0.193\textwidth}
\centering
\captionsetup{justification=centering}
\includegraphics[width=\textwidth, trim = {0 0 0 0}, clip]{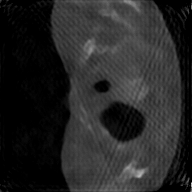}
\caption{\scalebox{0.95}[1.0]{Data-consistent}\\{PSNR $=34.10$}}
\label{fig:regular-DC-1}
\end{subfigure}
\caption{Sample for which the data-consistent PSNR value is relatively high compared to the U-Net PSNR values.}
\label{fig:regular-1}
\end{figure}

\begin{figure}[!ht]
\centering
\begin{subfigure}[t]{0.193\textwidth}
\centering
\captionsetup{justification=centering}
\includegraphics[width=\textwidth, trim = {0 0 0 0}, clip]{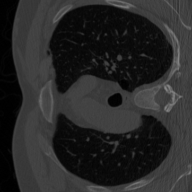}
\caption{\scalebox{0.95}[1.0]{Ground truth}}
\label{fig:regular-GT-2}
\end{subfigure}
\begin{subfigure}[t]{0.193\textwidth}
\centering
\captionsetup{justification=centering}
\includegraphics[width=\textwidth, trim = {0 0 0 0}, clip]{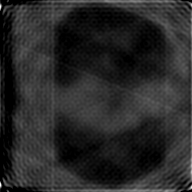}
\caption{\scalebox{0.95}[1.0]{Pseudo-inverse}\\{PSNR $=25.33$}}
\label{fig:regular-PI-2}
\end{subfigure}
\begin{subfigure}[t]{0.193\textwidth}
\centering
\captionsetup{justification=centering}
\includegraphics[width=\textwidth, trim = {0 0 0 0}, clip]{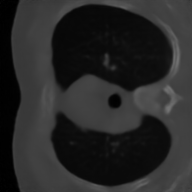}
\caption{\scalebox{0.95}[1.0]{One U-Net}\\{PSNR $=30.09$}}
\label{fig:regular-U1-2}
\end{subfigure}
\begin{subfigure}[t]{0.193\textwidth}
\centering
\captionsetup{justification=centering}
\includegraphics[width=\textwidth, trim = {0 0 0 0}, clip]{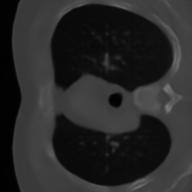}
\caption{\scalebox{0.95}[1.0]{Two U-Nets}\\{PSNR $=29.94$}}
\label{fig:regular-U2-2}
\end{subfigure}
\begin{subfigure}[t]{0.193\textwidth}
\centering
\captionsetup{justification=centering}
\includegraphics[width=\textwidth, trim = {0 0 0 0}, clip]{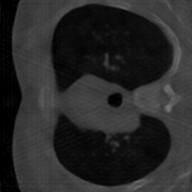}
\caption{\scalebox{0.95}[1.0]{Data-consistent}\\{PSNR $=30.09$}}
\label{fig:regular-DC-2}
\end{subfigure}
\caption{Sample for which the data-consistent PSNR value is approximately the same as the U-Net PSNR values.}
\label{fig:regular-2}
\end{figure}

\begin{figure}[!ht]
\centering
\begin{subfigure}[t]{0.193\textwidth}
\centering
\captionsetup{justification=centering}
\includegraphics[width=\textwidth, trim = {0 0 0 0}, clip]{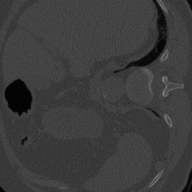}
\caption{\scalebox{0.95}[1.0]{Ground truth}}
\label{fig:regular-GT-3}
\end{subfigure}
\begin{subfigure}[t]{0.193\textwidth}
\centering
\captionsetup{justification=centering}
\includegraphics[width=\textwidth, trim = {0 0 0 0}, clip]{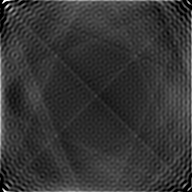}
\caption{\scalebox{0.95}[1.0]{Pseudo-inverse}\\{PSNR $=22.87$}}
\label{fig:regular-PI-3}
\end{subfigure}
\begin{subfigure}[t]{0.193\textwidth}
\centering
\captionsetup{justification=centering}
\includegraphics[width=\textwidth, trim = {0 0 0 0}, clip]{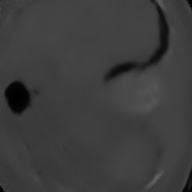}
\caption{\scalebox{0.95}[1.0]{One U-Net}\\{PSNR $=29.78$}}
\label{fig:regular-U1-3}
\end{subfigure}
\begin{subfigure}[t]{0.193\textwidth}
\centering
\captionsetup{justification=centering}
\includegraphics[width=\textwidth, trim = {0 0 0 0}, clip]{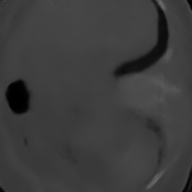}
\caption{\scalebox{0.95}[1.0]{Two U-Nets}\\{PSNR $=28.36$}}
\label{fig:regular-U2-3}
\end{subfigure}
\begin{subfigure}[t]{0.193\textwidth}
\centering
\captionsetup{justification=centering}
\includegraphics[width=\textwidth, trim = {0 0 0 0}, clip]{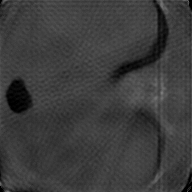}
\caption{\scalebox{0.95}[1.0]{Data-consistent}\\{PSNR $=24.52$}}
\label{fig:regular-DC-3}
\end{subfigure}
\caption{Sample for which the data-invariant PSNR value is relatively low compared to the U-Net PSNR values.}
\label{fig:regular-3}
\end{figure}

\newpage
\subsection{Modified test set}\label{app:modified}
The ground truth images in the modified test set contain more regions which are non-constant, as opposed to the regular test set, which consists of piecewise constant images. The data-consistent network is better able to deal with these modifications in the images, as can be seen particularly well in Figure \ref{fig:modified-4}: The U-Nets create a piecewise constant dark structure in the middle, while the data-consistent network keeps it more smooth. 
\begin{figure}[!ht]
\centering
\begin{subfigure}[t]{0.193\textwidth}
\centering
\captionsetup{justification=centering}
\includegraphics[width=\textwidth, trim = {0 0 0 0}, clip]{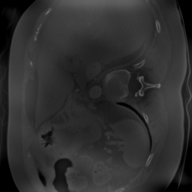}
\caption{\scalebox{0.95}[1.0]{Ground truth}}
\label{fig:modified-GT-4}
\end{subfigure}
\begin{subfigure}[t]{0.193\textwidth}
\centering
\captionsetup{justification=centering}
\includegraphics[width=\textwidth, trim = {0 0 0 0}, clip]{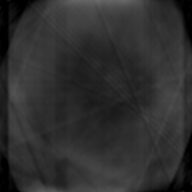}
\caption{\scalebox{0.95}[1.0]{Pseudo-inverse}\\{PSNR $=34.76$}}
\label{fig:modified-PI-4}
\end{subfigure}
\begin{subfigure}[t]{0.193\textwidth}
\centering
\captionsetup{justification=centering}
\includegraphics[width=\textwidth, trim = {0 0 0 0}, clip]{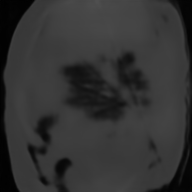}
\caption{\scalebox{0.95}[1.0]{One U-Net}\\{PSNR $=25.96$}}
\label{fig:modified-U1-4}
\end{subfigure}
\begin{subfigure}[t]{0.193\textwidth}
\centering
\captionsetup{justification=centering}
\includegraphics[width=\textwidth, trim = {0 0 0 0}, clip]{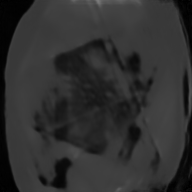}
\caption{\scalebox{0.95}[1.0]{Two U-Nets}\\{PSNR $=27.61$}}
\label{fig:modified-U2-4}
\end{subfigure}
\begin{subfigure}[t]{0.193\textwidth}
\centering
\captionsetup{justification=centering}
\includegraphics[width=\textwidth, trim = {0 0 0 0}, clip]{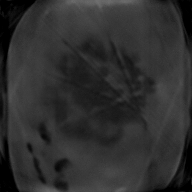}
\caption{\scalebox{0.95}[1.0]{Data-consistent}\\{PSNR $=33.61$}}
\label{fig:modified-DC-4}
\end{subfigure}
\caption{Sample for which the data-invariant PSNR value is relatively high compared to the U-Net PSNR values.}
\label{fig:modified-4}
\end{figure}

\begin{figure}[!ht]
\centering
\begin{subfigure}[t]{0.193\textwidth}
\centering
\captionsetup{justification=centering}
\includegraphics[width=\textwidth, trim = {0 0 0 0}, clip]{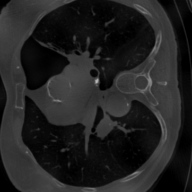}
\caption{\scalebox{0.95}[1.0]{Ground truth}}
\label{fig:modified-GT-5}
\end{subfigure}
\begin{subfigure}[t]{0.193\textwidth}
\centering
\captionsetup{justification=centering}
\includegraphics[width=\textwidth, trim = {0 0 0 0}, clip]{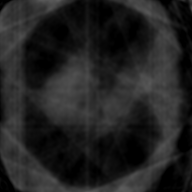}
\caption{\scalebox{0.95}[1.0]{Pseudo-inverse}\\{PSNR $=28.43$}}
\label{fig:modified-PI-5}
\end{subfigure}
\begin{subfigure}[t]{0.193\textwidth}
\centering
\captionsetup{justification=centering}
\includegraphics[width=\textwidth, trim = {0 0 0 0}, clip]{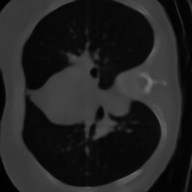}
\caption{\scalebox{0.95}[1.0]{One U-Net}\\{PSNR $=31.50$}}
\label{fig:modified-U1-5}
\end{subfigure}
\begin{subfigure}[t]{0.193\textwidth}
\centering
\captionsetup{justification=centering}
\includegraphics[width=\textwidth, trim = {0 0 0 0}, clip]{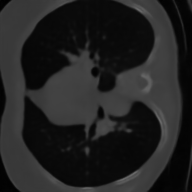}
\caption{\scalebox{0.95}[1.0]{Two U-Nets}\\{PSNR $=31.52$}}
\label{fig:modified-U2-5}
\end{subfigure}
\begin{subfigure}[t]{0.193\textwidth}
\centering
\captionsetup{justification=centering}
\includegraphics[width=\textwidth, trim = {0 0 0 0}, clip]{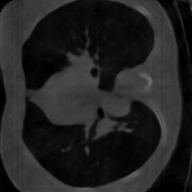}
\caption{\scalebox{0.95}[1.0]{Data-consistent}\\{PSNR $=31.47$}}
\label{fig:modified-DC-5}
\end{subfigure}
\caption{Sample for which the data-invariant PSNR value is approximately the same as the U-Net PSNR values.}
\label{fig:modified-5}
\end{figure}

\begin{figure}[!ht]
\centering
\begin{subfigure}[t]{0.193\textwidth}
\centering
\captionsetup{justification=centering}
\includegraphics[width=\textwidth, trim = {0 0 0 0}, clip]{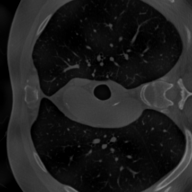}
\caption{\scalebox{0.95}[1.0]{Ground truth}}
\label{fig:modified-GT-6}
\end{subfigure}
\begin{subfigure}[t]{0.193\textwidth}
\centering
\captionsetup{justification=centering}
\includegraphics[width=\textwidth, trim = {0 0 0 0}, clip]{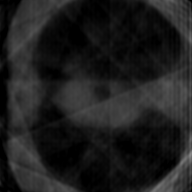}
\caption{\scalebox{0.95}[1.0]{Pseudo-inverse}\\{PSNR $=28.62$}}
\label{fig:modified-PI-6}
\end{subfigure}
\begin{subfigure}[t]{0.193\textwidth}
\centering
\captionsetup{justification=centering}
\includegraphics[width=\textwidth, trim = {0 0 0 0}, clip]{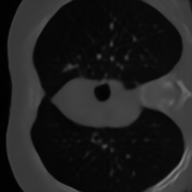}
\caption{\scalebox{0.95}[1.0]{One U-Net}\\{PSNR $=31.05$}}
\label{fig:modified-U1-6}
\end{subfigure}
\begin{subfigure}[t]{0.193\textwidth}
\centering
\captionsetup{justification=centering}
\includegraphics[width=\textwidth, trim = {0 0 0 0}, clip]{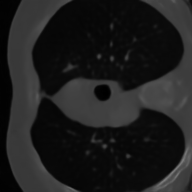}
\caption{\scalebox{0.95}[1.0]{Two U-Nets}\\{PSNR $=31.30$}}
\label{fig:modified-U2-6}
\end{subfigure}
\begin{subfigure}[t]{0.193\textwidth}
\centering
\captionsetup{justification=centering}
\includegraphics[width=\textwidth, trim = {0 0 0 0}, clip]{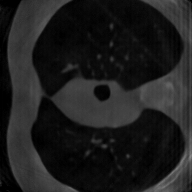}
\caption{\scalebox{0.95}[1.0]{Data-consistent}\\{PSNR $=30.89$}}
\label{fig:modified-DC-6}
\end{subfigure}
\caption{Sample for which the data-invariant PSNR value is relatively low compared to the U-Net PSNR values.}
\label{fig:modified-6}
\end{figure}

\end{document}